\documentclass[11pt,leqno]{article}

\usepackage{amsmath, amsfonts, theorem,color,
}
\usepackage[latin1]{inputenc}

\usepackage[T1]{fontenc}

\usepackage[english]{babel}

\usepackage{lmodern}

\usepackage{amsmath}

\usepackage{amssymb}

\usepackage{mathrsfs}
\usepackage{latexsym}
\usepackage{graphicx}
\makeatletter
\def\@maketitle{\newpage
    \null
    \vskip .8truein
    \begin{center}%
     {\bf \@title \par}%
     \vskip 1.5em
     {\small
      \lineskip .5em
      \begin{tabular}[t]{c}\@author
      \end{tabular}\par}%
    \end{center}%
    \par
    \vskip .4truein}
\@addtoreset{equation}{section} \@addtoreset{theorem}{section}
\@addtoreset{lemma}{section} \@addtoreset{proposition}{section}
\@addtoreset{definition}{section} \@addtoreset{corollary}{section}
\@addtoreset{remark}{section}

\def\dfrac#1#2{\ds{\frac{#1}{#2}}}

\newcommand{\re}{{\mathbb R}}

\let\ds=\displaystyle

\def\R{{\mathbb R}}

\newtheorem{theorem}{Theorem}[section]
\newtheorem{lemma}{Lemma}[section]
\newtheorem{proposition}{Proposition}[section]
\newtheorem{definition}{Definition}[section]
\newtheorem{corollary}{Corollary}[section]
\newtheorem{remark}{Remark}[section]
\newtheorem{example}{Example}[section]
\newtheorem{hypothesis}{Hypothesis}[section]

  {\hfill$\Box$\bigskip\par}

\DeclareMathOperator{\diver}{div}
\def\proof{\list{}{\setlength{\leftmargin}{0pt}
                      \parskip=0pt\parsep=0pt\listparindent=2em
                      \itemindent=0pt}\item[]\futurelet\testchar\@maybe}

\def\@maybe{\ifx[\testchar \let\next\@Opt
          \else \let\next\@NoOpt \fi \next}
\def\@Opt[#1]{{\it Proof of #1.\ }}\def\@NoOpt{{\it Proof.\ }}
\begin{document}
\title{\Large \bf Non-coercive first order Mean Field Games}
\author{{\large \sc Paola Mannucci\thanks{Dipartimento di Matematica ``Tullio Levi-Civita'', Universit\`a di Padova, mannucci@math.unipd.it}, Claudio Marchi \thanks{Dipartimento di Ingegneria dell'Informazione, Universit\`a di Padova, claudio.marchi@unipd.it},}\\
 {\large \sc Carlo Mariconda\thanks{Dipartimento di Matematica ``Tullio Levi-Civita'', Universit\`a di Padova, carlo.mariconda@unipd.it},  Nicoletta Tchou\thanks{Univ Rennes, CNRS, IRMAR - UMR 6625, F-35000 Rennes, France, nicoletta.tchou@univ-rennes1.fr}}%\\
% \rm Universit\`a degli Studi di Padova, Universit\'e de Rennes 1
}

%%%%%%%%%%%%%%%%%%%%%%%%%%%%%%%%%%%%%%%%%%%%%%%%%%%%%%%%%%%%%%%%%
\maketitle

\begin{abstract}
  We study first order evolutive Mean Field Games where the Hamiltonian is non-coercive. This situation occurs, for instance, when some directions are ``forbidden'' to the generic player at some points. We establish the existence of a weak solution of the system via a vanishing viscosity method and, mainly, we prove that the evolution of the population's density is the push-forward of the initial density through the flow characterized almost everywhere by the optimal trajectories of the control problem underlying the Hamilton-Jacobi equation. As preliminary steps, we need that the optimal trajectories for the control problem are unique (at least for a.e. starting points) and that the optimal controls can be expressed in terms of the horizontal gradient of the value function.
%To this end, a crucial ingredient will be the uniqueness of optimal trajectories.
\end{abstract}


%
%Introduction
%
\section{Introduction}

In this paper we study the following Mean Field Game (briefly, MFG)
\begin{equation}\label{eq:MFG1}
\left\{\begin{array}{lll}
(i)&\quad-\partial_t u+H(x, Du)=F(x,m)&\qquad \textrm{in }\re^2\times (0,T)\\
(ii)&\quad\partial_t m-\diver  (m\, \partial_pH(x, Du))=0&\qquad \textrm{in }\re^2\times (0,T)\\
(iii)&\quad m(x,0)=m_0(x), u(x,T)=G(x, m(T))&\qquad \textrm{on }\re^2,
\end{array}\right.
\end{equation}
where, if $p=(p_1,p_2)$ and $x=(x_1, x_2)$, the functions $H(x,p)$ is
\begin{equation}\label{H}
H(x,p)=\frac{1}{2}(p_1^2+h^2(x_1)p_2^2)
\end{equation}
where $h(x_1)$ is a regular bounded function %(s.t. $\vert h \vert _{C^{2,\delta}}<\infty$),
 {\em possibly vanishing} and that $F$ and $G$  are strongly regularizing (see assumptions (H1) -- (H4) below).

These MFG systems arise when the dynamics of the generic player are deterministic and,  when $h$ vanishes, may have a ``forbidden'' direction; actually, if the evolution of the whole population's distribution~$m$ is given, each agent wants to choose the control $\alpha=(\alpha_1, \alpha_2)$ in $L^2([t,T];\R^2)$
%\begin{equation}\label{A}
%{ \cal{A}}=\{\alpha:[t,T]\rightarrow \re^2, \alpha\in L^2([t,T]) \}
%\end{equation}
in order to minimize the cost
\begin{equation}\label{Jgen}
\int_t^T\left[\frac12 |\alpha(\tau)|^2+F(x(\tau),m)\right]\,d\tau+G(x(T), m(T))
\end{equation}
where, in $[t,T]$, its dynamics $x(\cdot)$ are governed by
\begin{equation}\label{eq:HJ2}
\left\{
\begin{array}{l}
 x_1'(s)=\alpha_1(s) \\
 x_2'(s)=h(x_1(s))\alpha_2(s)
\end{array}\right.
\end{equation}
with $x_1(t)=x_1$ and $x_2(t)=x_2$.
We see that the direction along $x_2$ is forbidden when $h(x_1)$ has zero value. This kind of problems are called of ``Grushin type'' (see \cite{LM} or Example~\ref{ex:Gru} below). As a matter of fact the structure of this degenerate dynamics will play an essential rule in our results.
Even though our techniques apply to a wider class of degenerate operators (see  the forthcoming papers \cite{MMMT2, MMMT3}), in the present paper we restrict our attention to this class of problems because they already contain all the main technical issues.

Let us recall that the MFG theory studies Nash equilibria in games with a huge number of (``infinitely many'') rational and indistinguishable agents. This theory started with the pioneering papers by Lasry and Lions \cite{LL1,LL2,LL3} and by Huang, Malham\'e and Caines \cite{HMC}. A detailed description of the achievements obtained in these years goes beyond the scope of this paper; we just refer the reader to the monographs \cite{AC, C, BFY, GPV, GS}.

As far as we know, degenerate MFG systems have been poorly investigated up to now. Dragoni and Feleqi~\cite{DF} studied a second order (stationary) system where the principal part of the operator fulfills the H\"ormander condition; moreover, Cardaliaguet, Graber, Porretta and Tonon~\cite{CGPT} tackled degenerate second order systems with coercive (and convex as well) first order operators. Hence, these results cannot be directly applied to the non-coercive problem~\eqref{eq:MFG1}.

 The aim of this paper is to prove the existence of a solution of \eqref{eq:MFG1}.
The main result is the interpretation of the evolution of the population's density as the push-forward of the distribution at the initial time through a flow which is suitably defined in terms of the optimal control problem underlying Hamilton Jacobi equation. %~\eqref{eq:MFG1}.

 In order to establish a representation formula for $m$, we shall follow some ideas of P-L Lions in the lectures at College de France (2012) (see \cite{C}), some results proved in
\cite{CH,C13} and the Ambrosio superposition principle \cite{AGS}.
Indeed the non-coercivity of~$H$ prevents from applying directly the arguments of \cite[Sect. 4.3]{C}.
Actually we have to study carefully the behaviour of the optimal trajectories of the control problem associated to the Hamilton-Jacobi equation~\eqref{eq:MFG1}-(i) especially their uniqueness. A crucial point will be the application of the Pontryagin maximum principle and the statement of Theorem~\ref{th:nobifurc} on the uniqueness of the optimal trajectory after the rest time. As far as we know this  uniqueness property has never been tackled before for this kind of degenerate dynamics and in our opinion may have interest in itself.

We point out that our approach could be applied to other first order ``degenerate'' MFG
but it is essential to prove some uniqueness properties of optimal trajectories in a set of starting points of full measure.
In general this set depends on the semiconcavity properties of $u$, as in the classical setting, and on the degeneracy of the dynamics.

We now list our notations and the assumptions, we give the definition of (weak) solution to system~\eqref{eq:MFG1} and we state the existence result for system~\eqref{eq:MFG1}.

\noindent\underline {Notations and Assumptions}.
For $x=(x_1, x_2)\in \re^2$, $\phi:\R^2\to\R$ and $\Phi:\R^2\to\R^2$  differentiable, we set: $D_G \phi(x):=(\partial_{x_1}\phi(x),h(x_1)\partial_{x_2}\phi(x))$ and $\diver_G\Phi(x):=\partial_{x_1}\Phi_1(x)+h(x_1)\partial_{x_2}\Phi_2(x).$
We denote by~$\mathcal P_1$ the space of Borel probability measures on~$\re^d$ with finite first order moment, endowed with the Kantorovich-Rubinstein distance~{${\bf d}_1$}. We denote $C^2(\R^2)$ the space of functions with continuous second order derivatives endowed with the norm $\|f\|_{C^2}:=\sup_{x\in\R^2}[|f(x)|+|Df(x)|+|D^2f(x)|]$. Throughout this paper, we shall require the following hypotheses:
\begin{itemize}
\item[(H1)]\label{H1} The functions~$F(\cdot,\cdot)$ and $G(\cdot,\cdot)$ are real-valued function, continuous on \\$\re^2\times\mathcal P_1$;
\item[(H2)]\label{H2} The map $F(x,\cdot)$ is Lipschitz continuous from $\mathcal P_1$ to $C^{2}(\re^2)$ uniformly for $x\in\R^2$;
%,  i.e.
%$$\|F(x,m_1)-F(x,m_2)\|_\infty\leq C {\bf d}_1(m_1,m_2);$$
moreover, there exists~$C\in \mathbb R$ such that
$$\|F(\cdot,m)\|_{C^2}, \|G(\cdot,m)\|_{C^2}\leq C,\qquad \forall m\in \mathcal P_1,$$
\item[(H3)]\label{H3} the function $h:\mathbb R\to\mathbb R$ is $C^{2}(\mathbb R)$
with $\|h\|_{C^2}\leq C$;
\item[(H4)]\label{H4} the initial distribution~$m_0$ has a compactly supported density (that we still denote by~$m_0$, with a slight abuse of notation), $m_0\in C^{2,\delta}(\re^2)$, for a $\delta\in (0,1)$.
\end{itemize}

\vskip 5mm

\begin{example}\label{ex:Gru}
%Examples of metric \eqref{eq:HJ2} are the Grushin type problems,
%% with  $C^2$ and bounded $h$, as
%with
Easy examples of $h$ are $h(x_1)=\sin (x_1)$ or $h(x_1)=\frac{x_1}{\sqrt{1+x_1^2}}$, (see \cite{LM} where the term $h(x_1)=\frac{x_1}{\sqrt{1+x_1^2}}$ is introduced as a degenerate diffusion term).
%all the points $s$ such that $h(s)=0$ are isolated).
%will satisfy all the hypotheses.
%(see \cite{LM}).
\end{example}

We now introduce our definition of solution of the MFG system~\eqref{eq:MFG1} and state the main result concerning its existence.
\begin{definition}\label{defsolmfg}
The pair $(u,m)$ is a solution of system~\eqref{eq:MFG1} if:
\begin{itemize}
\item[1)] $(u,m)\in W^{1,\infty}(\re^2\times[0,T])\times C^0([0,T];\mathcal P_1(\re^2))$;
\item[2)] Equation~\eqref{eq:MFG1}-(i) is satisfied by $u$ in the viscosity sense;
\item[3)] Equation~\eqref{eq:MFG1}-(ii) is satisfied by $m$ in the sense of distributions.
\end{itemize}
\end{definition}
Here below we state the main result of this paper.
\begin{theorem}\label{thm:main}
Under the above assumptions:
\begin{enumerate}
\item System \eqref{eq:MFG1} has a solution $(u,m)$ in the sense of Definition~\ref{defsolmfg},
%with $m\in C^0([0,T];\mathcal P_1)$;
\item $m$ is the push-forward of $m_0$ through the characteristic flow
\begin{equation}\label{chflow}
\left\{
\begin{array}{ll}
x_1'(s)=-u_{x_1}(x(s),s),& \quad x_1(0)=x_1, \\
x_2'(s)=-h^2(x_1(s))u_{x_2}(x(s),s),&\quad x_2(0)=x_2
\end{array}\right..
\end{equation}
\end{enumerate}
\end{theorem}

%We study the following Mean Field Game (briefly, MFG)
%\begin{equation}\label{eq:MFG1}
%\left\{\begin{array}{lll}
%(i)&\quad-\partial_t u+\frac 1 2 ( (\partial_{x_1}u)^2+h(x_1)^2(\partial_{x_2}u)^2)=F(x,t, m)&\qquad \textrm{in }\re^2\times (0,T)\\
%(ii)&\quad\partial_t m-div_G (m\, D_Gu)=0&\qquad \textrm{in }\re^2\times (0,T)\\
%(iii)&\quad m(x,0)=m_0(x), u(x,T)=G(x, m(T))&\qquad \textrm{on }\re^2,
%\end{array}\right.
%\end{equation}
%where, $p=(p_1,p_2)$, $x=(x_1, x_2)$, and $D_Gu=(\partial_{x_1}u, h(x_1)\partial_{x_2}u)$.

%In this case we easily see that the direction along $x_2$ is forbidden when $h(x_1)$ has zero value.
%\begin{example}\label{ex:Gru}
%Examples of metric defined by \eqref{GRU}, are the Grushin type problems, with  $C^2$ and bounded $h$, as $h(x_1)=\sin (x_1)$ or $h(x_1)=\frac{x_1}{\sqrt{1+x_1^2}}$ (see \cite{LM}).
%\end{example}
\begin{remark} Uniqueness holds under classical hypothesis on the monotonicity of $F$ and $G$ as in \cite{C}.
\end{remark}

This paper is organized as follows. In Section~\ref{OC} we will find some properties of the solution $u$ of the Hamilton-Jacobi equation~\eqref{eq:MFG1}-(i) with fixed $m$: we will prove that $u$ is Lipschitz continuous and semiconcave in $x$. Moreover still in this section we will establish a  crucial point of the paper: the uniqueness of the optimal trajectory of the associated control problem.
In Section~\ref{sect:c_eq} we study the continuity equation~\eqref{eq:MFG1}-(ii) where~$u$ is the solution of a Hamilton-Jacobi equation found in the previous section.
Section~\ref{sect:MFG} is devoted to the proof of the Theorem~\ref{thm:main}. % we give the proof of the existence of the solution to system~\eqref{eq:MFG1}.
Finally, the Appendix splits into two parts: in the former one, we give some results on the concatenation of optimal trajectories and the Dynamic Programming Principle while in the latter part we introduce the notion of the $G$-differentiability and we prove the main properties on the $G$-differentials which will be used along the paper.
%introduce the notion of $G$-differentiability with the main properties of $G$-differentiable functions.

\section{Formulation of the optimal control problem}\label{OC}
For every $0\le t<T$ and $ x:=({x}_1, {x}_2)\in\R^2$ we consider the following optimal control problem, where the functions $f,g,h$ satisfy the Hypothesis~\ref{BasicAss} here below.
\begin{definition}[Optimal Control Problem (OC)]\label{def:OCD}
\begin{equation}\label{def:OC}
\text{Minimize } J_t(x(\cdot),\alpha):
=\displaystyle\int_t^T\dfrac12|\alpha(s)|^2+f( x(s),s)\,ds+g(x(T))
\end{equation}
subject to $(x(\cdot), \alpha)\in \mathcal A(x,t)$, where
\begin{equation} \mathcal A(x,t):=\left\{(x(\cdot), \alpha)\in AC([t,T]; \R^2)\times L^2([t,T];\re^2):\, \textrm{\eqref{eq:HJ2} holds a.e. with } x(t)= x\right\}.
\end{equation}
A pair $(x(\cdot),\alpha)$ in $\mathcal A(x,t)$ is said to be admissible. We say that $x^*$ is an optimal trajectory if there is a control $\alpha^*$ such that $(x^*,\alpha^*)$ is optimal for~{\rm (OC)}. Also, we shall refer to the system~\eqref{eq:HJ2}
%\begin{equation}
%\begin{cases} x'_1=\alpha_1\\x_2'=h(x_1)\alpha_2\end{cases}\!\!\! \text{ a. e.}\label{tag:S}
%\end{equation}
as to the dynamics of the optimal control problem~{\rm (OC)}.
\end{definition}

In what follows, the functions $f,g$ and $h$ satisfy the following conditions.
\begin{hypothesis} \label{BasicAss}
$f\in C^{0}([0,T],C^2(\R^2))$ and %, $g\in C^{2}(\R^2)$, $h\in C^{2}(\R^2)$.
there exists a constant $C$ such that
$$\|f(\cdot,t)\|_{C^2(\R^2)} + \|g\|_{C^2(\R^2)} +  \|h\|_{C^2(\R)}\leq C,\qquad\forall t\in[0,T].$$
The set ${\cal{Z}}=\{z\in\R\;:\;h(z)=0\}$ is totally disconnected.
\end{hypothesis}
\begin{remark}
Our Hypothesis \ref{BasicAss} is very strong and we clearly do not need it all over
the paper. We much prefer to state a more restrictive result without the need to add further assumptions step by step.
\end{remark}

\begin{remark} Thanks to the boundedness of functions in our Hypothesis \ref{BasicAss}, it is equivalent to choose $ \mathcal {A}(x,t)$ or
\begin{equation}
\tilde{\mathcal{A}}(x,t):=\left\{(x(\cdot), \alpha)\in AC([t,T]; \R^2)\times L^1([t,T];\re^2):\,\textrm{$x$ satisfies \eqref{eq:HJ2} a.e.},\,
x(t)= x \right\},
\end{equation}
in the control problem.
\end{remark}
\begin{remark}\label{rem:uniquex} Notice that, given a control law $\alpha\in L^2([t,T];\R^2)$, the above Hypothesis~\ref{BasicAss} on $h$ implies that, given the initial point $ x$,  there is a \emph{unique} trajectory $x(\cdot)$ such that $(x(\cdot),\alpha)\in\mathcal A(x,t)$.
\end{remark}
\begin{remark} [Existence of optimal solutions] Hypothesis~\ref{BasicAss}, together with   \cite[Theorem 23.11]{Cla}, ensure that the optimal control problem~{\rm (OC)} admits a solution $(x^*,\alpha^*)$.
\end{remark}

\subsection{The Hamilton-Jacobi equation and the value function of the optimal control problem}\label{vf}
The aim of this section is to study the Hamilton-Jacobi equation~\eqref{eq:MFG1}-(i) with $m$ fixed, namely
\begin{equation}\label{eq:HJ1}
\left\{\begin{array}{ll}
-\partial_t u+\frac12 |D_Gu|^2=f(x, t)&\qquad \textrm{in }\re^2\times (0,T),\\
u(x,T)=g(x)&\qquad \textrm{on }\re^2
\end{array}\right.
\end{equation}
where~$f(x,t):=F(x, m)$ and $g(x):=G(x, m(T))$.
Under Hypothesis 2.1, we shall prove several regularity properties of the solution (especially Lipschitz continuity and semiconcavity) and mainly the uniqueness of optimal trajectories for the associated optimal control problem.
\begin{definition} The value function for the cost $J_t$ defined in \eqref{def:OC} is
\begin{equation}\label{repr}u(x,t):=\inf\left\{ J_t(x(\cdot), \alpha):\, (x(\cdot),\alpha)\in \mathcal A(x,t)\right\}.
\end{equation}
An optimal pair $(x^*(\cdot), \alpha^*)$ for the control problem {\rm (OC)} in Definition \ref{def:OCD} is also said to be optimal for $u(x,t)$.
\end{definition}
In the next lemma we show that the solution $u$ of \eqref{eq:HJ1} can be represented as the value function of the control problem {\rm (OC)} defined in \eqref{def:OCD}.
\begin{lemma}\label{valuefunction}
Under Hypothesis 2.1, the value function $u$, defined in~\eqref{repr}, is the unique bounded uniformly continuous viscosity solution to problem~\eqref{eq:HJ1}.
\end{lemma}
\begin{proof}
The Dynamic Programming Principle (stated in Proposition~\ref{proposition:carloclaudio} in Appendix below) yields that the value function is a solution to problem~\eqref{eq:HJ1}. Moreover applying classical results uniqueness (see, for example, \cite[eq. (7.40) and Thm. 7.4.14]{CS}), we obtain the statement.
Moreover, taking as admissible control the law $\alpha=0$, from the representation formula \eqref{repr}, using the boundedness of $f$ and $g$, we have $\vert u(x,t)\vert \leq C_T$.
\end{proof}

\begin{lemma}[Lipschitz continuity] \label{L1}
Under Hypothesis 2.1, there hold:
\begin{enumerate}
%\item The function $u$ defined in \eqref{repr} is bounded in $\re^2\times [0,T]$,
\item $u(x,t)$ is Lipschitz continuous with respect to the spatial variable~$x$,
\item $u(x,t)$ is Lipschitz continuous with respect to the time variable~$t$.
\end{enumerate}
\end{lemma}
\begin{proof}  In this proof, $C_T$ will denote a constant which may change from line to line but it always depends only on the constants in the assumptions (especially the Lipschitz constants of $f$ and $g$) and on~$T$.\\
1. Let $t$ be fixed. We follow the proof of \cite[Lemma 4.7]{C}.
Let $\alpha^\varepsilon$ be an $\varepsilon$-optimal control for $u(x,t)$ i.e.,
\begin{equation}
\label{eq:HJ31}
u(x_1, x_2, t)+\varepsilon\geq \int_t^T\frac12 |\alpha^{\varepsilon}(s)|^2+f(x(s),s)\,ds+g(x(T))
\end{equation}
where~$ x(\cdot)$ obeys to the dynamics \eqref{eq:HJ2} with $\alpha=\alpha^{\varepsilon}$.
%\begin{equation}
%\label{eq:HJ3}
%\forall s\in [t,T]\quad\left\{\begin{array}{ll}
%x_1'(s)=\alpha^{\varepsilon}_1(s),& \quad x_1(t)=x_1, \\
%5x_2'(s)=h(x_1(s))\alpha_2^{\varepsilon}(s),&\quad x_2(t)=x_2.
%\end{array}\right.
%\end{equation}
From the boundedness of $u$ (established in Lemma~\ref{valuefunction}) and our assumptions, there exists a constant~$C_T$ such that $\Vert \alpha^{\varepsilon}\Vert_{L^2(t,T)}\leq C_T$.\\
We consider the path $x^*(s)$ starting from $y=(y_1,y_2)$, with control $\alpha^{\varepsilon}(\cdot)$.
Hence
\begin{eqnarray*}
x_1^*(s)&=&y_1+\int_t^s\alpha^{\varepsilon}_1(\tau) \,d\tau=y_1-x_1+x_1(s)\\
x_2^*(s)&=&y_2+\int_t^s h(y_1-x_1+x_1(\tau))\alpha_2^{\varepsilon}(\tau)\,d\tau\\
&=&y_2-x_2+x_2(s)
+\int_t^s h(y_1-x_1+x_1(\tau))\alpha_2^{\varepsilon}(\tau)-h(x_1(\tau))\alpha_2^{\varepsilon}(\tau)\,d\tau.
\end{eqnarray*}
Using the Lipschitz continuity of $f$ and $h$ and the boundedness of $h$  we get
\begin{eqnarray*}
&&f(x^*(s),s)=\\
%&&f\bigg(y_1-x_1+x_1(s),y_2-x_2+x_2(s)+\\
%&&\phantom{AAAAA}+\int_t^s h(y_1-x_1+x_1(\tau))\alpha_2^{\varepsilon}(\tau)-h(x_1(\tau))\alpha_2^{\varepsilon}(\tau)\,d\tau,s\bigg)\leq\\
&& \leq f(x_1(s),x_2(s),s)
+L\vert y_1-x_1\vert+\\
&&\phantom{AAAAA}+L\left\vert y_2-x_2+\int_t^s h(y_1-x_1+x_1(\tau))\alpha_2^{\varepsilon}(\tau)-h(x_1(\tau))\alpha_2^{\varepsilon}(\tau)\,d\tau\right\vert\\
&&
\leq f(x_1(s),x_2(s),s)
+L\vert y_1-x_1\vert+L\vert y_2-x_2\vert+L'\vert y_1-x_1 \vert \int_t^s \vert\alpha_2^{\varepsilon}(\tau)\vert \,d\tau\\
&&
\leq f(x(s),s)
+L\vert y_1-x_1\vert+L\vert y_2-x_2\vert+L'\vert y_1-x_1 \vert T^{\frac 1 2}  \bigg(\int_t^s(\alpha_2^{\varepsilon}(s))^2 \,ds\bigg)^{\frac 1 2}.
\end{eqnarray*}
By the same calculations for $g$ and substituting inequality~\eqref{eq:HJ31} in
\begin{equation*}%\label{eq:HJ5}
u(y_1, y_2, t)\leq \int_t^T\frac12 |\alpha^{\varepsilon}(s)|^2+f(x^*(s),s)\,ds+g(x^*(T)),
\end{equation*}
we get
\begin{equation*}%\label{eq:HJ6}
u(y_1, y_2, t)\leq u(x_1, x_2, t)+C_T(\vert y_2-x_2\vert+\vert y_1-x_1 \vert).
\end{equation*}
Reversing the role of $x$ and $y$ we get the result.\\
2.
We follow the same arguments as  those in the proof of \cite[Lemma 4.7]{C}; to this end, we recall that $h$ is bounded and we observe that there holds
$$ |x(s)-x|\leq C(s-t)\|\alpha\|_{\infty}$$ since
$\alpha$ is bounded  as proved in Corollary \ref{coro:regularity} in the next section.
\end{proof}
%\begin{remark} {\huge C.: ma questo remark serve veramente?}
%We used the boundedness of  $f, g $ to prove the boundedness of $u$ and the uniform Lipschitz-continuity (and the boundedness) of
%$f, g $ and $h$ to prove the Lipschitz-continuity of $u$.
%\end{remark}

In the following lemma we establish the semiconcavity of $u(x,t)$ taking advantage of the regularity hypothesis \ref{BasicAss}. This property will be needed in the study of the relationship between the regularity of the value function and the uniqueness of the optimal trajectories. It is worth to remark that it is possible to prove that $u(x,t)$ is also semiconcave with respect to the $\chi$-lines associated to the Grushin dynamics, as introduced in \cite[Example 2.4]{BD}, but this does not seem to be useful to our results.

\begin{lemma}[Semi-concavity] Under Hypothesis 2.1, the value function~$u$, defined in \eqref{repr}, is semiconcave with respect to the variable~$x$.
\end{lemma}
\begin{proof}
For any~$x,y\in \re^2$ and $\lambda\in[0,1]$, consider~$x_{\lambda}:=\lambda x+(1-\lambda)y$. Let $\alpha$ be an~$\varepsilon$-optimal control for~$u(x_{\lambda}, t)$; we set
$$x_{\lambda}(s)= (x_{\lambda,1}(s), x_{\lambda,2}(s)):=
\bigg(x_{\lambda,1}+\int_t^s\,\alpha_1(\tau)\,d\tau,\  x_{\lambda,2}+\int_t^s\,h(x_{\lambda,1}(\tau))\alpha_2(\tau)\,d\tau\bigg).$$
Let $x(s)$ and $y(s)$ satisfy \eqref{eq:HJ2} with initial condition respectively $x$ and $y$ still with the same control $\alpha$, $\varepsilon$-optimal for~$u(x_{\lambda}, t)$.
We have to estimate~$\lambda u(x,t) +(1-\lambda)u(y,t)$ in terms of $u(x_{\lambda}, t)$. To this end, arguing as in the proof of~\cite[Lemma 4.7]{C},  we have to estimate the terms $\lambda f(x(s), s) +(1-\lambda)f(y(s), s)$ and $\lambda g(x(T))+(1-\lambda)g(y(T)).$\\
We explicitly provide the calculations for the second component $x_2(s)$ since the calculations for $x_1(s)$ are the same as in \cite{C}.
We have
\begin{eqnarray*}
x_2(s)&=& x_2+\int_t^s h(x_1(\tau))\alpha_2(\tau)\,d\tau\\
&=&x_2-x_{\lambda,2}+x_{\lambda,2}(s)+\int_t^s\,(h(x_{1}(\tau))-h(x_{\lambda,1}(\tau)))\alpha_2(\tau)\,d\tau,
\end{eqnarray*}
and analogously for $y_2(s)$.
For the sake of brevity we provide the explicit calculations only for $f$ and we omit the analogous ones for $g$; we write $f(x_1,x_2):=f(x_1, x_2, s)$.
We have
\begin{equation*}
\begin{array}{l}
\lambda f(x(s))+(1-\lambda)f(y(s))=\\
\qquad \lambda f\left(x_1(s),x_{\lambda,2}(s)+x_2-x_{\lambda,2}+\displaystyle\int_t^s\,
(h(x_{1}(\tau))-h(x_{\lambda,1}(\tau)))\alpha_2(\tau)\,d\tau\right)\\
\qquad+(1-\lambda)f\left(y_1(s), x_{\lambda,2}(s)+y_2-x_{\lambda,2}+\displaystyle\int_t^s\,(h(y_{1}(\tau))-h(x_{\lambda,1}(\tau)))\alpha_2(\tau)\,d\tau\right).
\end{array}
\end{equation*}
In the Taylor expansion of $f$ centered in  $x_{\lambda}(s)$ the contribution of the first variable can be dealt with as in \cite{C}. Assuming without any loss of generality $x_1=y_1$, the contribution of the second variable gives
\begin{multline*}
\lambda f(x(s))+(1-\lambda)f(y(s))=
f(x_{\lambda}(s))+ \partial_{x_2}f(x_{\lambda}(s))\bigg(\lambda(x_2-x_{\lambda,2})+(1-\lambda)(y_2-x_{\lambda,2})\\
+\lambda\int_t^s\,(h(x_{1}(\tau))-h(x_{\lambda,1}(\tau)))\alpha_2(\tau)\,d\tau)\\
+(1-\lambda)\int_t^s\,(h(y_{1}(\tau))-h(x_{\lambda,1}(\tau)))\alpha_2(\tau)\,d\tau)\bigg)+R,
\end{multline*}
where $R$ is the error term of the expansion, namely
\begin{multline}\label{R}
R=\lambda \frac{\partial_{x_2,x_2}^2f(\xi_1)}{2}\bigg(x_2- x_{\lambda,2}+ \int_t^s\,(h(x_{1}(\tau))-h(x_{\lambda,1}(\tau)))\alpha_2(\tau)\,d\tau\bigg)^2\\
+(1- \lambda) \frac{\partial_{x_2,x_2}^2f(\xi_2)}{2}\bigg(y_2- x_{\lambda,2}+ \int_t^s\,(h(y_{1}(\tau))-h(x_{\lambda,1}(\tau)))\alpha_2(\tau)\,d\tau\bigg)^2,
\end{multline}
for suitable $\xi_1, \xi_2\in \re^2$.\\
Since
$\lambda(x_2-x_{\lambda,2})+(1-\lambda)(y_2-x_{\lambda,2})=0$, we get
\begin{equation}\label{convf}
\lambda f(x(s))+(1-\lambda)f(y(s))=f(x_{\lambda}(s))
+ \partial_{x_2}f(x_{\lambda}(s))\int_t^s\,I(\tau)\alpha_2(\tau)\,d\tau+R,
\end{equation}
with $I(\tau):=-h(x_{\lambda,1}(\tau))+\lambda h(x_{1}(\tau)) +(1-\lambda)h(y_{1}(\tau)).$ Now, our aim is to estimate~$I(\tau)$.
Since $x_{\lambda,1}(\tau)=\lambda x_1(\tau)+(1-\lambda)y_1(\tau)$, $x_{1}(\tau)-x_{\lambda,1}(\tau)= (1-\lambda)(x_1-y_1)$ and $y_{1}(\tau)-x_{\lambda,1}(\tau)=\lambda(y_1-x_1)$, the Taylor expansion for $h$ centered in $x_{\lambda,1}(\tau)$ yields
$$I(\tau) =\frac{1}{2}  (1-\lambda)\lambda(y_{1}-x_{1})^2[ (1-\lambda)h^{''}(\overline \xi)+\lambda h^{''}(\widetilde{\xi})],
$$
for suitable $\overline{\xi}, \widetilde{\xi}\in\re$.
Our Hypothesis 2.1 entails
$$|I(\tau)| \leq (1-\lambda)\lambda C(y_1-x_{1})^2.$$
Replacing the inequality above in~\eqref{convf}, we obtain
\begin{equation}\label{stimaf}
\lambda f(x_2(s))+(1-\lambda)f(y_2(s))\leq \\
f(x_{\lambda,2}(s))+
 C^2T(1-\lambda)\lambda (y_{1}-x_{1})^2+ R.
\end{equation}
Let us now estimate the error term $R$ in \eqref{R}.
We have
\begin{eqnarray*}
&&\bigg(x_2- x_{\lambda,2}+ \int_t^s\,(h(x_{1}(\tau))-h(x_{\lambda,1}(\tau)))\alpha_2(\tau)\,d\tau\bigg)^2
\\
&&\qquad\leq 2(x_2- x_{\lambda,2})^2+ 2\bigg(\int_t^s\,(h(x_{1}(\tau))-h(x_{\lambda,1}(\tau)))\alpha_2(\tau)\,d\tau\bigg)^2\\
&& \qquad\leq2{(1-\lambda)}^2(x_2-y_2)^2+ 2C{(1-\lambda)}^2(x_1-y_1)^2\leq C(1-\lambda)^2|x-y|^2
\end{eqnarray*}
and, analogously
\begin{eqnarray*}
  &&\bigg(y_2- x_{\lambda,2}+ \int_t^s\,(h(y_{1}(\tau))-h(x_{\lambda,1}(\tau)))\alpha_2(\tau)\,d\tau\bigg)^2 \leq  C\lambda^2|x-y|^2
\end{eqnarray*}
Then, replacing these two inequalities in~\eqref{R}, we infer
\begin{equation}\label{resto}
R\leq C(1-\lambda) \lambda|x-y|^2.
\end{equation}
Taking into account~\eqref{resto} and~\eqref{stimaf}, we get the semiconcavity of $u$.
\end{proof}

%\subsection{Uniqueness of the trajectories}
\subsection{Necessary conditions and regularity for the optimal trajectories}
The application of the Maximum Principle yields the following necessary conditions.
\begin{proposition}[Necessary conditions for optimality]\label{prop:pontriagin}
\label{MPP}
Let $(x^*, \alpha^*)$ be optimal for {\rm (OC)}. There exists an arc $p\in AC([t,T];\R^2)$, hereafter called the costate,  such that
\begin{enumerate}
\item The pair $(\alpha^*, p)$ satisfies
the {\it adjoint equations}: for a. e. $s\in[t,T]$
\begin{eqnarray}
&&p_1'=-p_2h'(x_1^*)\alpha_2^*+f_{x_1}( x^*, s)\label{tag:adjoint1}\\
&&p_2'=f_{x_2}( x^*,s),\label{tag:adjoint2}
\end{eqnarray}
the {\it transversality condition}
\begin{equation}\label{tag:transversality}
-p(T)=D g(x^*(T))
\end{equation}
together with the {\it maximum condition}
\begin{multline}\label{tag:max}
\max_{\alpha=(\alpha_1,\alpha_2)\in\R^2}p_1(s)\alpha_1+p_2(s)h(x_1^*(s))\alpha_2-\dfrac{|\alpha|^2}2
=\\=
p_1(s)\alpha_1^*(s)+p_2(s)h(x_1^*(s))\alpha_2^*(s)-\dfrac{|\alpha^*(s)|^2}2\text{ a. e.  }s\in [t,T].\end{multline}
\item The optimal control $\alpha^*$ is given by
\begin{equation}\begin{cases}
\alpha^*_1=p_1\\
\alpha^*_2=p_2h(x^*_1)\end{cases}\text{ a.e on }[t,T].\label{tag:alpha*}
\end{equation}
\item The pair
$(x^*, p)$ satisfies the system of differential equations: for a. e. $s\in[t,T]$
\begin{eqnarray}
&&x'_1=  p_1\label{tag:1} \\
&&x_2'= h^2(x_1)p_2 \label{tag:2}\\
&&p_1'=  -p_2^2h'(x_1)h(x_1)+f_{x_1}(x,s)\label{tag:3}\\
&&p_2'=f_{x_2}(x,s)\label{tag:4}
\end{eqnarray}
with the mixed boundary conditions $x^*(t)= x$, $p(T)=-D g(x^*(T))$.
\end{enumerate}
\end{proposition}
\begin{proof} 1. Hypothesis~\ref{BasicAss} ensures the validity of the assumptions of the Maximum Principle \cite[Theorem 22.17]{Cla}. Since the endpoint is free,  \cite[Corollary 22.3]{Cla} implies that the deduced necessary conditions  hold in normal form: the claim follows directly.

2. The maximum condition \eqref{tag:max} implies that
\[D_{\alpha}\left(p_1(s)\alpha_1+p_2(s)h(x_1^*(s))\alpha_2-\dfrac{|\alpha|^2}2\right)_{\alpha=\alpha^*}=0\quad \text{a. e. }s\in [t, T]\]
from which we get \eqref{tag:alpha*}.

3. Conditions \eqref{tag:1} -- \eqref{tag:2} follow directly from the dynamics~\eqref{eq:HJ2} replacing $\alpha_1^*, \alpha_2^*$ by means of \eqref{tag:alpha*}. Condition~\eqref{tag:3} follows similarly from \eqref{tag:adjoint1}, whereas   \eqref{tag:4} coincides with \eqref{tag:adjoint2}.
\end{proof}
\begin{remark}\label{rem:fixedendpoint} In the case where one prescribes the endpoint $x(T)$  in Definition \ref{def:OCD}, the proof of Proposition~\ref{prop:pontriagin} shows that the claim still holds true without the endpoint condition for $p(T)$ in Point 3.
\end{remark}

\begin{corollary}[Feedback control and regularity]\label{coro:regularity}
Let $(x^*, \alpha^*)$ be optimal for $u(x,t)$ and $p$ be the  related costate as in Proposition~\ref{prop:pontriagin}. Then:
\begin{enumerate}
\item The costate $p=(p_1, p_2)$ is uniquely expressed in terms of $x^*$ for every $s\in [t, T]$ by
\begin{equation}
\!\begin{cases}\label{tag:p}
p_1(s)\!\!&\!\!\!=-g_{x_1}(x^*(T))-\!\!\displaystyle\int_s^T \!\!f_{x_1}(x^*,\tau)-p_2^2h'(x_1^*)h(x_1^*)\,d\tau,\\
p_2(s)\!\!&\!\!\!=-g_{x_2}(x^*(T))-\displaystyle\int_s^Tf_{x_2}(x^*,\tau)\,d\tau.\\
\end{cases}
\end{equation}
\item The optimal control
 $\alpha^*=(\alpha^*_1, \alpha^*_2)$ is a feedback control {\rm (}i.e., a function of $x^*${\rm )}, uniquely expressed   in terms of $x^*$ for a. e. $s\in [t, T]$ by
\begin{equation}\!\begin{cases}\label{tag:alpha}
\alpha_1^*(s)\!\!&\!\!\!=-g_{x_1}(x^*(T))+\!\!\displaystyle\int_T^s \!\!f_{x_1}(x^*,\tau)-p_2^2h'(x_1^*)h(x_1^*)\,d\tau,
\\
\alpha_2^*(s)\!\!&\!\!\!=p_2(s) h(x_1^*(s)).
\end{cases}
\end{equation}
\item  The optimal trajectory $x^*$ and the optimal control $\alpha^*$ are of class $C^1$. % and $\alpha^*$ is absolutely continuous. In particular the equalities in \eqref{tag:alpha*}, \eqref{tag:1} -- \eqref{tag:2} and \eqref{tag:alpha} do hold for every $s\in [t, T]$.
In particular the equalities \eqref{tag:alpha*} -- \eqref{tag:alpha} do hold for every $s\in [t, T]$.
\item Assume that,  for some $k\in\mathbb N$,
 $h\in C^{k+1}$ and  $D_xf(x,s)$ is of class $C^k$.
Then $x^*, p$ and $\alpha^*$  are of class   $C^{k+1}$.
\end{enumerate}
\end{corollary}
\begin{proof} Point 1 is an immediate consequence of \eqref{tag:3} -- \eqref{tag:4} together with the endpoint condition $p(T)=-D g(x^*(T))$. Point 2 follows then directly from \eqref{tag:alpha*}.

3. Since $x^*$ is continuous, the continuity of  $\alpha^*$ follows from \eqref{tag:alpha}. The dynamics~\eqref{eq:HJ2} then imply that $x^*\in C^1$.
%Since $D_xf(x,s)$ is continuous, r
Relations~\eqref{tag:p} and~\eqref{tag:alpha}  imply, respectively,  that  $p$ and $\alpha^*$  are of class $C^1$.
%The absolute continuity of $\alpha_1^*$ follows from the equality $\alpha_1^*=p_1$ in \eqref{tag:alpha*}, whereas that of $\alpha_2^*$ can be deduced from the second equation in \eqref{tag:alpha}.

4.
%Assume that $h\in C^{2}$ and that $D_xf(x,s)$ is of class $C^1$.
The relations~\eqref{tag:p} and the $C^1$-regularity of $x^*$ and $p$  imply that, actually,  $p\in C^2$. Therefore, \eqref{tag:alpha} gives the $C^2$-regularity of $\alpha^*$ and, finally, the dynamics \eqref{eq:HJ2} yield the $C^2$-regularity of $x^*$.
Further regularity of $x^*$, $\alpha^*$ and $p$ follows by a standard bootstrap inductive argument.
\end{proof}

\subsection{Uniqueness of the trajectories after the initial time}
Next Theorem~\ref{th:nobifurc} implies that the optimal trajectories for $u(x,t)$, do not bifurcate at any time $r>t$ whenever $h( x_1)\not=0$ (see Corollary \ref{cor:nobifurc}), otherwise they may rest at $ x$ in an interval from the initial time $t$ but they do not bifurcate as soon as they leave $ x$.
\begin{theorem}[Uniqueness of the optimal trajectory after the rest time]\label{th:nobifurc}
Under Hypothesis \ref{BasicAss}, let $x^*$ be an  optimal trajectory for $u(x,t)$.
\begin{enumerate}
\item Assume that $h(x_1^*(\tau))\not=0$ for some $t<\tau<T$. For every $\tau\le  r< T$ there are no other  optimal trajectories for $u(x^*(r),r)$, other than $x^*$, restricted to $[r,T]$.
\item\label{ii-th:nobifurc} Assume that $h( x_1)=0$. Let $t_{x^*}$ be the rest time for $x^*$ defined by
\[t_{x^*}:=\sup\{r\in[t, T]:\, x^*\equiv x\text{ on }[t, r]\}.\]
For every $r>t_{x^*}$ there are no  optimal trajectories for $u(x^*(r),r)$,  other  than $x^*$ restricted to $[r,T]$.
\end{enumerate}
\end{theorem}
\begin{remark} We point out that the rest time may be positive only when $h(x_1)$ vanishes. Notice also that $t_{x^*}=T$ if and only if $x^*$ is constant on $[t, T]$.
\end{remark}
The next Lemma~\ref{lemma:singular} relates the initial constancy of an optimal trajectory to a stationary condition and is a key argument of the proof of part \eqref{ii-th:nobifurc} of Theorem~\ref{th:nobifurc}.
\begin{lemma}[A stationary condition]\label{lemma:singular} Assume that $h( x_1)=0$. Let $x^*=(x^*_1, x^*_2)$ be an optimal trajectory for  $u(x,t)$, and $r\in [t, T]$. Then
\begin{equation}\label{tag:singular}x^*\equiv x\text{ on } [t, r]\Leftrightarrow h(x_1^*)\equiv 0\text{ on } [t,r].\end{equation}
\end{lemma}
\begin{proof} If $h(x_1^*)=0$ on $[t,r]$ then $x_1^*$ belongs to set of the zeros  of $h$, which is totally disconnected by Hypothesis~\ref{BasicAss}. It follows that $x_1^*\equiv x_1$ on $[t,r]$. Moreover, the dynamics
\eqref{eq:HJ2} imply that $x_2^*\equiv  x_2$  on $[t, r]$, so that $x^*\equiv  x$ on $[t, r]$.
The opposite implication is trivial, since $h( x_1)=0$.
\end{proof}
\begin{remark} The two  conditions in  \eqref{tag:singular} are both equivalent to the fact that the (unique) control $\alpha^*$ such that $J_t(x^*, \alpha^*)=u( x, t)$  is singular on $[t,r]$ in the sense of \cite[Def. 2.3]{CR}, i.e.,  that there exists an absolutely continuous arc   $\pi:[t, r]\rightarrow \R^2\setminus\{(0,0)\}$ satisfying
\begin{equation}\label{tag:SO1}
\pi_1=0,\quad \pi_2h(x_1^*)=0,\quad \pi_1'=-\pi_2^2h'(x_1^*)h(x_1^*), \quad \pi_2'=0.
 \end{equation}
Indeed, if $x^*\equiv  x$ on $[t,r]$ for some $t<r\le T$ then any arc $\pi:=(0, c)$ with $0\not=c\in\R$ is such that $(x^*, \pi)$ satisfies \eqref{tag:SO1} on $[t, r]$.
Conversely, if $(x^*, \pi)$ fulfills \eqref{tag:SO1} on $[t, r]$ for some $t<r\le T$, then $\pi_2=c$ for some constant $c$. Since $\pi_1=0$ then $c\not=0$. It follows from the second equation in \eqref{tag:SO1}  that $h(x_1^*)=0$ on $[t,r]$.\\
Notice that \eqref{tag:SO1} are the conditions satisfied by the costate $p$ in \eqref{tag:alpha*} and \eqref{tag:3} -- \eqref{tag:4} when $\alpha^*=(0,0)$ and $f\equiv 0$.
\end{remark}
\begin{proof}[Theorem~\ref{th:nobifurc}] 1.
Let $y^*$ be optimal  for $u(x^*(r),r)$.
Point 1 of Proposition~\ref{proposition:carloclaudio} in the Appendix ensures that the concatenation $z^*$ of $x^*$ with $y^*$ at $r$ is optimal for $u(x,t)$. Let $p:=(p_1, p_2), q:=(q_1, q_2)$ be the costates associated to  $x^*:=(x^*_1, x^*_2)$ and, respectively, to $z^*:=(z^*_1, z^*_2)$.
Both $(x^*, p)$ and $(z^*, q)$ satisfy \eqref{tag:1} -- \eqref{tag:4} on $[t, T]$.
Now, Corollary~\ref{coro:regularity} shows that $x^*$ and $z^*$ are of class $C^1$.
Since $x^*=z^*$ on $[t, \tau]$, the fact that $\tau>t$, together with  \eqref{tag:1} imply
\[p_1(\tau)=(x_1^*)'(\tau)=\lim_{s\to\tau^-}(x_1^*)'(s)=\lim_{s\to\tau^-}(z_1^*)'(s)=(z_1^*)'(\tau)=q_1(\tau),\]
whereas \eqref{tag:2}, and the fact that $h(x_1^*(\tau))\not=0$  analogously  yield
\[p_2(\tau)=\dfrac{(x_2^*)'(\tau)}{h^2(x_1^*(\tau))}=\dfrac{(z_2^*)'(\tau)}{h^2(z_1^*(\tau))}=q_2(\tau).\]
Therefore, both $(x^*, p)$ and $(z^*, q)$ are absolutely continuous solutions to  the same Cauchy problem on $[t, T]$, with Cauchy data at $\tau$, for the first order differential system \eqref{tag:1}-\eqref{tag:4}. The regularity assumptions on $f$ and $h$ and Caratheodory's Theorem guarantee the uniqueness of the solution. Thus $x^*=z^*$ on $[t,T]$, from which we obtained the desired equality $x^*=y^*$ on $[r, T]$.

2. We assume that $t_{x^*}<T$, otherwise the claim is trivial.
We deduce from Lemma~\ref{lemma:singular} that there is $\tau\in [t_{x^*},r]$ satisfying $h(x_1^*(\tau))\not=0$.
If $y^*$ is optimal for $u(x^*(r),r)$, then  Point 1 of  Proposition~\ref{proposition:carloclaudio} shows that the concatenation $z^*$ of $x^*$ with $y^*$ at $r$ is optimal for $u(x,t)$.
Moreover, Point 2 of Proposition~\ref{proposition:carloclaudio} imply that both $x^*$ and $z^*$, restricted to $[\tau, T]$, are  optimal for $u(x^*(\tau), \tau)$. Point 1 of Theorem~\ref{th:nobifurc}  implies that $x^*=y^*$ on $[\tau, T]$, proving the desired result.
\end{proof}
\begin{corollary}\label{cor:nobifurc}
%Assume that the derivative  $h'$ of $h$  and  $D_x f(\cdot,s)$ are locally Lipschitz for every $s\in [t,T]$.
Let $x^*$ be an  optimal trajectory for $u(x,t)$. If $h(x_1)\not=0$, for every $0<  r< T$ there are no other  optimal trajectories for $u(x^*(r),r)$, other than $x^*$, restricted to $[r,T]$.
\end{corollary}
\section{The continuity equation}\label{sect:c_eq}
In this section we want to study equation \eqref{eq:MFG1}-(ii).
Since $h$ is independent of $x_2$, taking account of \eqref{H},
this partial differential equation can be rewritten as
\begin{equation}\label{cont}
\partial_t m-\partial_{x_1} (m \partial_{x_1}u)-h^2(x_1)\partial_{x_2} (m \partial_{x_2}u)=\partial_t m-\diver _G (m D_G u)=0.
\end{equation}
Hence our aim is to study the well posedness of the problem
\begin{equation}\label{continuity}
\left\{
\begin{array}{ll} \partial_t m-
%\diver  (m\, \partial_pH(x, Du))=0
\diver_G (m\, D_G u)=0,
&\qquad \textrm{in }\re^2\times (0,T),\\
m(x,0)=m_0(x), &\qquad \textrm{on }\re^2,
\end{array}\right.
\end{equation}
where $u$ is a solution to problem
\begin{equation}\label{HJ}
\left\{\begin{array}{ll}
-\partial_t u+\frac12 |D_Gu|^2=F(x, \overline {m})&\qquad \textrm{in }\re^2\times (0,T),\\
u(x,T)=G(x, \overline {m}(T)),&\qquad \textrm{on }\re^2,
\end{array}\right.
\end{equation}
where the function~$\overline m$ is fixed and fulfills
\begin{equation}\label{mcnd}
\overline {m}\in C^{1/2}([0,T],\mathcal P_1),\qquad \int_{\re^2}|x|^2\,d\overline m(t)(x)\leq K, \qquad t\in[0,T].
\end{equation}
Note that this problem is equivalent to \eqref{eq:HJ1} with a fixed $\overline m$.
% and, for $x=(x_1, x_2)\in \re^2$, $\phi:\R^2\to\R$ and $\Phi:\R^2\to\R^2$  differentiable, we set $$D_G \phi(x):=(\partial_{x_1}\phi(x) , h(x_1)\partial_{x_2}\phi(x)),\quad \diver_G\Phi(x) :=\partial_{x_1}\Phi_1(x)+h(x_1)\partial_{x_2}\Phi_2(x).$$
%\begin{remark}
%It is worth noting that, since $h$ is independent of $x_2$, the partial differential equation in \eqref{continuity} can be rewritten as
%\begin{equation}\label{cont}
%\partial_t m-\partial_{x_1} (m \partial_{x_1}u)-h^2(x_1)\partial_{x_2} (m \partial_{x_2}u)=\partial_t m-\diver _G (m D_G u)=0.
%\end{equation}
%%where we recall that
%%$D_G \phi(x):=(\partial_{x_1}\phi(x),h(x_1)\partial_{x_2}\phi(x))$ and $\diver_G\Phi(x):=\partial_{x_1}\Phi_1(x)+h(x_1)\partial_{x_2}\Phi_2(x).$
%\end{remark}

Observe that, by Lemma \ref{L1}-(1), in~\eqref{continuity} the {\it drift}~$v=D_Gu$ is only bounded; this lack of regularity prevents to apply the standard results (uniqueness, existence and representation formula of $m$ as the push-forward of $m_0$ through the characteristic flow; e.g., see \cite[Proposition 8.1.8]{AGS}) for drifts which are Lipschitz continuous in~$x$. We shall overcome this difficulty applying Ambrosio superposition principle \cite[Theorem 8.2.1]{AGS} and proving several results on the uniqueness of the optimal trajectory for the control problem stated in Section \ref{OC}. The Ambrosio superposition principle yields a representation formula of $m$ as the push-forward of some measure on~$C^0([0,T],\re^2)$ through the evaluation map~$e_t$. In the following theorem, we shall also recover uniqueness, existence and some regularity result for the solution to~\eqref{continuity}.

\begin{theorem}\label{prp:m}
Under assumptions {\rm (H1) -- (H4)}, for any $\overline m$ as in \eqref{mcnd},
Problem \eqref{continuity} has a unique bounded solution $m$ in the sense of Definition \ref{defsolmfg}. Moreover $m(t,\cdot)$ is absolutely continuous with ~$\sup_{t\in[0,T]}\|m(t,\cdot)\|_{\infty}\leq C$ and it is a Lipschitz continuous map from $[0,T]$ to $\mathcal P_1$ with a Lipschitz constant bounded by $\|Du\|_\infty \|h^2\|_\infty$.
Moreover, the function $m$ satisfies:
\begin{equation}
\label{ambrosio}
\int_{\re^2} \phi\, dm(t)=\int_{\re^2}\phi(\overline {\gamma}_x(t))\,m_0(x)\, dx \qquad \forall \phi\in C^0_0(\R^2), \, \forall t\in[0,T]
\end{equation}
where, for a.e. $x\in\re^2$,  $\overline{\gamma}_x$ is the solution to \eqref{chflow}.
\end{theorem}

The proof of  Theorem~\ref{prp:m} is given in the next two subsections which are devoted respectively to the existence result (see Proposition~\ref{VV}), to the uniqueness result and to the representation formula (see Proposition~\ref{!FP}) and to the Lipschitz regularity (see Corollary~\ref{lemma:m_lip}).
\subsection{Existence of the solution}\label{subsect:ex}

As in \cite[Appendix]{C13} (see also \cite[Section 4.4]{C}), we now want to establish the existence of a solution to the continuity equation %MFG system
via a vanishing viscosity method, applied on the {\it whole} MFG system.
\begin{proposition}\label{VV}
Under assumptions {\rm (H1) -- (H4)},
problem \eqref{continuity} has a bounded solution $m$ in the sense of Definition \ref{defsolmfg}.
\end{proposition}

We consider the solution $(u^\sigma, m^\sigma)$ to the following problem
\begin{equation}
\label{eq:MFGv}
\left\{
\begin{array}{lll}
&(i)\quad-\partial_t u-\sigma \Delta u+\frac12 |D_G u|^2
%  \frac 1 2 \left((\partial_1u)^2+h(x_1)^2(\partial_2u)^2\right)
=F(x, \overline {m})&\qquad \textrm{in }\re^2\times (0,T)\\
&(ii)\quad \partial_t m-\sigma \Delta m-\diver _G (m D_G u)=0&\qquad \textrm{in }\re^2\times (0,T)\\
&(iii)\quad m(x,0)=m_0(x), u(x,T)=G(x, {\overline {m}}(T))&\qquad \textrm{on }\re^2.
\end{array}\right.
\end{equation}
Let us recall that equation \eqref{eq:MFGv}-(ii) has a standard interpretation in terms of a suitable stochastic process (see relation~\eqref{mstoch} below).
Our aim is to find a solution to problem~\eqref{continuity} letting $\sigma \to 0^+$. To this end some estimates are needed; as a first step, we establish the well-posedness of system~\eqref{eq:MFGv}.

\begin{lemma}\label{visco:buonapos}
Under assumptions {\rm (H1) -- (H4)}, for any $\overline m$ as in \eqref{mcnd}, there exists a unique bounded classical solution $(u^\sigma, m^\sigma)$ to problem~\eqref{eq:MFGv}. Moreover, $m^\sigma>0$.
\end{lemma}
\begin{proof}
The proof uses standard regularity results for quasilinear parabolic equations;
from Lemma \ref{visco:lemma5.2} here below, the solution $u^\sigma$ of \eqref{eq:MFGv}-(i) is bounded in $\R^2\times[0,T]$.
Hence we can apply \cite[Theorem 8.1, p.495]{LSU}, obtaining the existence and uniqueness of a classical solution $u^\sigma$ in all $\R^2\times[0,T]$.
Now $m^\sigma$ is the classical solution of the linear equation
\[
\quad \partial_t m-\sigma \Delta m+b\cdot D m+c_0m=0,%\tilde{f},
\quad m(0)=m_0
\]
with $b$ and $c_0$ H\"older continuous coefficients.
Hence still applying classical results (see \cite[Theorem 5.1, p.320]{LSU}) we get the existence and uniqueness of a classical solution $m^\sigma$ of \eqref{eq:MFGv}-(ii).
From assumptions on $m_0$ and the maximum principle (see for example \cite[Theorem 2.1, p.13]{LSU}) we get that $m^\sigma>0$.
\end{proof}

Let us now prove that the functions $u^\sigma$ are Lipschitz continuous and semiconcave uniformly in~$\sigma$.

\begin{lemma}\label{visco:lemma5.2}
Under the same assumptions of Lemma~\ref{visco:buonapos}, there exists a constant $C>0$, independent of $\sigma$ such that
\[\|u^\sigma\|_\infty\leq C,\quad
\|D u^\sigma\|_\infty\leq C\quad\textrm{and}\quad D^2 u^\sigma\leq C\qquad \forall \sigma>0.
\]
\end{lemma}
\begin{proof}
The $L^\infty$-estimate easily follows from the Comparison Principle and assumption (H2) because the functions $w^\pm:= C\pm C(T-t)$ are respectively a super- and a subsolution for \eqref{eq:MFGv}-(i) if $C$ is sufficiently large.

We refer to \cite{C}  for the proof of the uniform Lipschitz continuity of the functions $u^\sigma$. The proof is similar to the deterministic one proved in Lemma \ref{L1} and it uses the representation formula by means a stochastic optimal control problem:
\[
u^\sigma(x,t)=\min \int_t^T\left[\frac12 |\alpha(\tau)|^2+f(Y(\tau),\tau)\right]\,d\tau+g(Y(T))
\]
where, in $[t,T]$,  $Y(\cdot)$ is governed by a stochastic differential equation
\begin{equation}
\label{1-stoc}
\left\{
\begin{array}{l}
dY_1=\alpha_1(t) dt +\sqrt{2\sigma} dB_{1,t}\\
dY_2= h(Y_1(t))\alpha_2(t) dt +\sqrt{2\sigma} dB_{2,t}
\end{array}
\right.,
\end{equation}
where  $Y(t)=x$ and $B_t$ is a standard 2-dimensional Brownian motion. (For an analytic proof see also \cite[Chapter XI]{Lie})

Let us now prove the part of the statement concerning the semiconcavity. We shall adapt the methods of \cite[Lemma 5.2]{C13}. We fix a direction $v=(\alpha_1,\alpha_2)$ with $|v|=1$ and compute the derivative of equation~\eqref{eq:MFGv}-(i) twice with respect to $v$ obtaining
\begin{equation*}
\begin{array}{l}
-\partial_t \partial_{vv}u-\sigma \Delta \partial_{vv} u-\partial_{vv}(F(x, \overline m(x,t))=-\partial_{vv}\left[\frac 1 2 \left((\partial_1u)^2+h(x_1)^2(\partial_2u)^2\right)\right]\\ \qquad
=-(D_G\partial_v u)^2-D_Gu\cdot D_G\partial_{vv} u-\frac12 \partial_{vv}(h^2)(\partial_2u)^2 - 4hh'\alpha_1\partial_2 u\partial_{2v}u\\ \qquad
\leq -(D_G\partial_v u)^2-D_Gu\cdot D_G\partial_{vv} u+ C(1+|D_G\partial_v u|)
\end{array}
\end{equation*}
(the last inequality is due to our assumptions and to the first part of the statement).
Since $-(D_G\partial_v u)^2+ C(1+|D_G\partial_v u|)$ is bounded above by a constant, we deduce
\[
-\partial_t \partial_{vv}u-\sigma \Delta \partial_{vv} u +D_Gu\cdot D_G\partial_{vv} u\leq C;\]
on the other hand, we have $\|\partial_{vv}u(T,\cdot)\|_\infty\leq C$ by assumption (H2) and we can conclude by comparison that $\partial_{vv}u\leq C'$ for a constant $C'$ independent of $\sigma$.
\end{proof}

Let us now prove some useful properties of the functions~$m^\sigma$.
\begin{lemma}\label{visco:lemma4}
Under the same assumptions of Lemma~\ref{visco:buonapos}, there exists a constant $K>0$, independent of $\sigma$, such that:
\begin{equation*}
\begin{array}{ll}
1.\quad &\|m^\sigma\|_\infty\leq K, \\
2.\quad &{\bf d}_1(m^\sigma(t_1)-m^\sigma(t_2))\leq K(t_2-t_1)^{1/2} \qquad \forall t_1,t_2\in(0,T),\\
3.\quad & \displaystyle\int_{\re^2}|x|^2\,dm^\sigma(t)(x)\leq K \left(\displaystyle\int_{\re^2}|x|^2\,dm_0(x)+1 \right)\qquad \forall t\in(0,T).
\end{array}
\end{equation*}
\end{lemma}
%{\it Claudio: il punto (iii) potrebbe essere generalizzato anche a $|x|^k$ con $k>2$, o mi sbaglio? Inoltre non lo si puo' ottenere con qualche stima della soluzione fondamentale?}
\begin{proof}
1. In order to prove this $L^\infty$ estimate, we shall argue as in \cite[Appendix]{C13}; for simplicity, we drop the $\sigma$'s. We note that
\[
\diver _G (m D_G u)=D_Gm\cdot D_Gu +m(\partial_{11}u +h^2 \partial_{22}u)\leq D_Gm\cdot D_Gu +Cm
\]
because of the semiconcavity of~$u$ established in Lemma~\ref{visco:lemma5.2} yields $\partial_{ii}u\leq C$ for $i=1,2$ (see \cite[Proposition1.1.3-(e)]{CS}) and $m\geq 0$. Therefore, by assumption (H2) the function~$m$ satisfies
\[
\partial_t m-\sigma \Delta m\leq D_Gm\cdot D_Gu +Cm,\qquad m(x,0)\leq C;
\]
using $w=Ce^{ C t}$ as supersolution (recall that $C$ is independent of $\sigma$), we infer: $\|m\|_\infty\leq w=Ce^{ C T}$.\\
To prove Points 2 and 3  as in the proof of \cite[Lemma 3.4 and 3.5]{C}, it is expedient to introduce the stochastic differential equation
\begin{equation}\label{11C}
dX_t= b(X_t,t) dt +\sqrt{2\sigma} dB_t,\qquad X_0=Z_0
\end{equation}
where $b=(\frac{\partial u^\sigma}{\partial x_1},h^2 \frac{\partial u^\sigma}{\partial x_2})$,
$B_t$ is a standard 2-dimensional Brownian motion, and ${\mathcal L}(Z_0)=m_0$. By standard arguments, (see \cite{Kr} and \cite[Chapter 5]{KS})
\begin{equation}
\label{mstoch}
m(t):={\mathcal L}(X_t)
\end{equation}
is a weak solution to~\eqref{eq:MFGv}-(ii).

The rest of the proof of Points 2 and 3 follows the same arguments of \cite[Lemma 3.4]{C} and, respectively, of \cite[Lemma 3.5]{C}; therefore, we shall omit it and we refer to \cite{C} for the detailed proof.
\end{proof}

Let us now prove that the  $u^\sigma$'s are uniformly bounded and uniformly continuous in time.

\begin{lemma}\label{visco:lemma5}
Under the same assumptions of Lemma~\ref{visco:buonapos}, the function~$u^{\sigma}$ is uniformly continuous in time uniformly in $\sigma$.
\end{lemma}
\begin{proof}
We shall follow the arguments in \cite[Theorem 5.1 (proof)]{C13}.
Let $u^\sigma_f:=u^\sigma(x,T)$; recall that, by assumption (H2),  $u^\sigma_f$ are bounded in $C^2$ uniformly in $\sigma$. Moreover, again by assumption {\rm(H2)}, there exists a constant $C_1$ sufficiently large such that the functions $\omega^{\pm}=u^\sigma_f(x)\pm C_1(T-t)$ are respectively super- and subsolution of \eqref{eq:MFGv}-(i) for any $\sigma$; actually, for $C_1=2C$ we have
\[
-\partial_t \omega^+-\sigma \Delta \omega^++\frac 1 2|D_G \omega^+|^2-F(x, \overline m)\geq C_1-\sigma C-C\geq 0
\]
and similarly for $\omega^-$.
Hence from the comparison principle we get
\begin{equation}\label{duestelle}
\|u^{\sigma}(x,t)-u^\sigma_f(x)\|_\infty\leq C_1(T-t)\qquad \forall t\in [0,T].
\end{equation}
We look now the source term $F(x,\overline m)$ of \eqref{eq:MFGv}-(i).
The Lipschitz continuity of $F$ w.r.t. $m$ (see assumption (H2)) and the H\"older continuity of $\overline m$ (see assumption~\eqref{mcnd}) imply:
$$
\sup_{t\in[h,T]}\|F(x,\overline m(t))-F(x,\overline m(t-h))\|_\infty\leq C \sup_{t\in[h,T]} {\bf d}_1(\overline m(t),\overline m(t-h))%\leq \omega(C|h|^{1/2})
=:\eta(h).
$$
The function $v_h^{\sigma}(x,t):= u^{\sigma}(x, t-h)+C_1h+\eta(h)(T-t)$ satisfies
\begin{multline*}
-\partial_t v_h^{\sigma}(x,t) -\sigma \Delta v_h^{\sigma}(x,t)+\frac 1 2|D_G v_h^{\sigma}(x,t)|^2-F(x, \overline {m})(x,t)+\eta(h)\\
= F(x, \overline m)(x,t-h)-F(x,\overline m)(x,t)+\eta(h)\geq 0\qquad \forall t\in[h,T]
\end{multline*}
and also $v_h^{\sigma}(x,T)= u^{\sigma}(x, T-h)+C_1h\geq u^{\sigma}(x, T)$ by estimate~\eqref{duestelle}; therefore, again by comparison principle, we get $
u^{\sigma}(x, t-h)+C_1h+\eta(h)(T-t)\geq u^{\sigma}(x, t)$. In a similar way we also obtain $u^{\sigma}(x, t-h)-C_1h-\eta(h)(T-t)\leq u^{\sigma}(x, t)$ accomplishing the proof.
\end{proof}

\begin{proof}[Proposition \ref{VV}]
We shall follow the proof of \cite[Theorem 5.1]{C13} (see also \cite[Theorem 4.20]{C}). We observe that, for all $\sigma\in(0,1)$, $m^\sigma$ belongs to  $C^0([0,T],\mathcal K)$ where $\mathcal K:=\{\mu\in \mathcal P_1:\, \textrm{$\mu$ satisfies Point 3 of Lemma~\ref{visco:lemma4}}\}$; moreover, we recall from \cite[Lemma 5.7]{C} that $\mathcal K$ is relatively compact in $\mathcal P_1$.\\
Lemma~\ref{visco:lemma5.2} and Lemma~\ref{visco:lemma5} imply that $u^\sigma$ uniformly converge to some function~$u$ and by standard stability result for viscosity solutions, the function~$u$ solves \eqref{HJ}, $u$ is Lipschitz continuous in~$x$, $Du^\sigma\to Du$ a.e. (because of the semiconcavity estimate of Lemma~\ref{visco:lemma5.2} and \cite[Theorem 3.3.3]{CS}), so, in particular, $D_G u^\sigma\to D_G u$ a.e.. \\
By the bounds on $m^\sigma$ contained respectively in Points 1 and 2 of  Lemma~\ref{visco:lemma4}, we obtain that, possibly passing to a subsequence, as $\sigma \to 0^+$, $m^\sigma$ converge to some $m\in C^0([0,T],\mathcal K)$ in the $C^0([0,T],\mathcal P_1)$ topology and in $L^\infty_{loc}((0,T)\times\re^2)$-weak-$*$ topology. Moreover we deduce that $m(0)=m_0$.
On the other hand, since $m^\sigma$ is a solution to \eqref{eq:MFGv}-(ii), for any $\psi\in C^\infty_0((0,T)\times\re^2)$, there holds
\[
\int_0^T\int_{\re^2}m^\sigma\left(-\partial_t \psi -\sigma \Delta \psi+D\psi\cdot D_Gu^\sigma \right)\,dx\, dt=0;
\]
letting $\sigma \to 0^+$, by the $L^\infty_{loc}$-weak-$*$ convergence of~$m^\sigma$ and by the convergence a.e. $D_G u^\sigma\to D_G u^\sigma$, we conclude that the function~$m$ solves \eqref{continuity}.
\end{proof}

\begin{remark}
As a matter of facts, we proved that the solution~$m$ to problem~\eqref{continuity} %belongs to the space $C^0([0,T];\mathcal P_1(\R^2))$ and it
fulfills the estimates in Lemma~\ref{visco:lemma4}.
\end{remark}

\subsection{Uniqueness of the solution}\label{uniq}

This section is devoted to establish the following uniqueness result for problem~\eqref{continuity}.
\begin{proposition}\label{!FP}
Under assumptions {\rm (H1) -- (H4)},
problem \eqref{continuity} admits at most one bounded solution $m$.
%in the space $C^0([0,T];\mathcal P_1(\R^2))$.
Moreover, the function $m$ satisfies:
\begin{equation}\label{ambrosio2}
\int_{\re^2} \phi\, dm(t)=\int_{\re^2}\phi(\overline{ \gamma}_x(t))\,m_0(x)\, dx, \qquad \forall \phi\in C^0_0(\R^2), \, \forall t\in[0,T]
\end{equation}
where, for a.e. $x\in\re^2$,  $\overline{\gamma}_x$ is the solution to \eqref{chflow}.

\end{proposition}

In order to prove this result, it is expedient to establish some properties of the optimal trajectories for the control problem defined in Section \ref{OC} and of the value function $u(x,t)$, defined in Subsection \ref{vf}.
For any $(x,t) \in \R^2\times[0,T]$, let $ \mathcal U(x,t)$ be the set of the optimal controls of the minimization problem {\rm (OC)} in Definition~\ref{def:OCD}.
We refer the reader to Appendix B, for the precise definition of $G$-differentiability and for its properties.
%\begin{remark}\label{AA}%(punto A mail Nicoletta)
%We recall that Corollary \ref{cor:nobifurc} ensures that the optimal trajectories do not bifurcate after the initial time if the initial point $x$ is such that $h(x_1)\not=0$.
%\end{remark}

\begin{lemma}\label{4.9} The following properties hold:
\begin{enumerate}
\item
 $D_Gu(x ,t)$ exists if and only if  $\alpha(t)$ is the same value for any $\alpha(\cdot)\in {\cal{U}}(x,t)$.
Moreover $D_Gu(x,t)=-\alpha(t)$ (i.e., $u_{x_1}(x,t)=-\alpha_1(t)$, $h(x_1(t))u_{x_2}(x,t)=-\alpha_2(t)$).
\item In particular, if $\mathcal U(x,t)$ is a singleton then $D_Gu(x(s),s)$ exists for any $s\in [t,T]$
where $x(s)$ is the optimal trajectory associated to the singleton of $\mathcal U(x,t)$.
\item If $x$ is such that $h(x_1)\neq 0$ and  $D_Gu(x,t)$ exists then there is a unique optimal trajectory starting from $x$ and $D_Gu(x,t)=-\alpha(t)$
% i.e.$-\alpha_1(t)=\partial_{x_1} u(x,t)$ e  $-\alpha_1(t)=h(x_1)\partial_{x_2} u(x,t)$
and hence
\begin{equation}\label{3.5.3}
x'_1(t)=-\partial_{x_1} u(x,t),\qquad x'_2(t)=-h^2(x_1)\partial_{x_2} u(x,t).
\end{equation}
\end{enumerate}\end{lemma}
\begin{proof}
1. We prove that  if $D_Gu(x ,t)$ exists then for any $\alpha(\cdot)\in {\cal{U}}(x,t)$ we have that $\alpha(t)$ is unique and $D_Gu(x,t)=-\alpha(t)$.
For any  $\alpha(\cdot)\in {\cal{U}}(x,t)$, let $x(\cdot)$ be the corresponding optimal trajectory. Then
$x(\cdot)$ and $\alpha(\cdot)$ satisfy the necessary conditions for optimality proved in Proposition
\ref{prop:pontriagin}.
Take $v=(v_1,v_2)\in\R^2$ and consider the solution $y(\cdot)$ of \eqref{eq:HJ2} with initial condition $y(t)=(x_1+v_1, x_2+h(x_1)v_2)$ and control~$\alpha$, namely
\begin{eqnarray*}
y_1(s)&=&x_1+v_1+\int_t^s\,\alpha_1(\tau)d\tau=x_1(s)+v_1,\\
y_2(s)&=&x_2+h(x_1)v_2+\int_t^s\,h(y_1(\tau))\alpha_2(\tau)d\tau\\
&=&x_2(s)+h(x_1)v_2 +
\int_t^s\,[h(y_1(\tau))-h(x_1(\tau))]\alpha_2(\tau)d\tau.
\end{eqnarray*}
Hence there holds
\begin{multline*}
u(x_1+v_1,x_2+h(x_1)v_2,t)-u(x_1,x_2,t)\leq\\
\int_t^T\left[ f\left(x_1(s)+v_1, x_2(s)+h(x_1)v_2 +
\int_t^s\,[h(y_1(\tau))-h(x_1(\tau))]\alpha_2(\tau)d\tau\right)- f(x_1(s), x_2(s))\right]ds+\\
g(y(T))-g(x(T)).
\end{multline*}
For $v=t(\hat v_1, \hat v_2)$ with $|(\hat v_1, \hat v_2)|=1$ and $t\in\R^+$, as $t\to 0^+$, the $G$-differentiability of $u$ at $(x,t)$ entails
\begin{equation*}
D_G u(x,t)\cdot (\hat v_1, \hat v_2)\leq (I_1,I_2)\cdot (\hat v_1, \hat v_2)
\end{equation*}
where
\begin{eqnarray*}
I_1&:=& \int_t^T\, f_{x_1}(x(s))ds+\int_t^T \bigg(f_{x_2}(x(s))\int_t^s\,h'(x_1(\tau))\alpha_2(\tau)d\tau\bigg)ds+g_{x_1}(x(T))+ \\
&& g_{x_2}(x(T))\int_t^T\,h'(x_1(\tau))\alpha_2(\tau)d\tau\\
I_2&:=& h(x_1)\left(\int_t^T\, f_{x_2}(x(s))ds+g_{x_2}(x(T))\right).
\end{eqnarray*}
By the arbitrariness of $(\hat v_1, \hat v_2)$, we get
\begin{eqnarray*}
D_G u(x,t)&=& (I_1,I_2).
\end{eqnarray*}
By \eqref{tag:4} and~\eqref{tag:transversality}, we obtain
\begin{eqnarray*}
I_1&=&\int_t^T\, f_{x_1}(x(s))ds+ \int_t^T \, (p_2^{\prime}(s)\int_t^s\,h^{\prime}(x_1(\tau))\alpha_2(\tau)d\tau)ds+g_{x_1}(x(T))\\
&&\qquad -p_2(x(T))\int_t^T\,h'(x_1(\tau))\alpha_2(\tau)d\tau\\
&=&\int_t^T\, f_{x_1}(x(s))ds-\int_t^T \, p_2(s)h^{\prime}(x_1(s))\alpha_2(s)ds+g_{x_1}(x(T))\\
&=&-\alpha_1(t)\end{eqnarray*}
where the last inequality is due to~\eqref{tag:p} and \eqref{tag:alpha*}.
On the other hand, again by \eqref{tag:p} and \eqref{tag:alpha*}, we have
\begin{equation*}
I_2= -h(x_1)p_2(t)=-\alpha_2(t).
\end{equation*}
The last three equalities imply: $D_G u(x,t)=-\alpha(t)$ which uniquely determines the value of $\alpha(\cdot)$ at time~$t$.

Conversely we prove that, if for any $\alpha(\cdot)\in {\cal{U}}(x,t)$,
$\alpha(t)$ is unique then  $D_Gu(x,t)$ exists.
%we have that  $\alpha(t)=-D_Gu(x,t)$
To prove the G-differentiability of $u(\cdot, t)$ in $x$, by the semiconcavity of $u$, we need to prove that $D_G^*u(x,t)$ is a singleton (see Theorem~\ref{thm336} in Appendix B below).
%(see Appendix, Section \ref{sect:Gdiff}, Proposition \ref{prp:gdiff}, Aggiungere in Appendix definizione di $D_G^*u(x,t)$).\\
Let $\pi\in D_G^*u(x,t)$. By definition of $D_G^*u(x,t)$ there exist two sequences $\{x_n\}$, $\{\pi_n=D_Gu(x_n,t)\}$ such that
\begin{equation}\label{1DG}
x_n\to x,\quad \pi_n\to\pi.
\end{equation}
Consider $\alpha_n\in\mathcal U(x_n,t)$;  by the other part of the statement (already proven), we know that
\begin{equation}\label{2DG}
-\alpha_n(t)=D_Gu(x_n,t)=\pi_n.
\end{equation}
From the definition of the cost $J$ (see Section \ref{OC}), using the optimality of $\alpha_n$ and the boundedness of the data we get
\begin{equation}\label{3DG}
\|\alpha_n\|_{L^2} \leq C,\  \text {for any } n.
\end{equation}
Let $x_n$ be the trajectory associated to $\alpha_n$, namely
\begin{equation*}
x_{n1}(s)=x_{n1}+\int_t^s\,\alpha_{n1}(\tau)d\tau,\qquad
x_{n2}(s)=x_{n2}+\int_t^s\,h(x_{n1}(\tau))\alpha_{n2}(\tau)d\tau.
\end{equation*}
From \eqref{3DG} and the boundedness of $h$, there exists a constant $C$ (independent of $n$) such that
\begin{equation}\label{4DG}
\|x_{n1}\|_{\infty}+\|x_{n2}\|_{\infty}\leq C,\ \text {for any } n.
\end{equation}
Let $(p_{n1}, p_{n2})$ be  the costate of $x_n$ as in Proposition \ref{prop:pontriagin}, using \eqref{tag:p} and~\eqref{4DG}, we get
\begin{equation}\label{5DG}
\|p_{n2}\|_{\infty}\leq C,\  \text {for any } n,
\end{equation}
and from \eqref{5DG}
\begin{equation}\label{5bisDG}
\|p_{n1}\|_{\infty}\leq C,\ \text {for any } n.
\end{equation}
Using \eqref{tag:alpha}:
\begin{equation}\label{6DG}
\|\alpha_{n1}\|_{\infty}+\|\alpha_{n2}\|_{\infty} \leq C,\  \text {for any } n.
\end{equation}
From Point (4) of Corollary \ref{coro:regularity} we can differentiate \eqref{tag:alpha*}, and using \eqref{tag:3}-\eqref{tag:4} we get:
\begin{eqnarray*}
\alpha_{n1}^{\prime}(s)&=&p_{n1}^{\prime}(s)= -p_{n2}^{2}(s)h^{\prime}(x_{n1}(s))h(x_{n1}(s))+f_{x_1}(x_{n1}(s),s),\\
\alpha_{n2}^{\prime}(s)&=& p_{n2}(s)h^{\prime}(x_{n1}(s))x_{n1}^{\prime}(s)+ p_{n2}^{\prime}(s)h(x_{n1}(s))\\
&=&
p_{n2}(s)h^{\prime}(x_{n1}(s))\alpha_{n1}(s)+f_{x_2}(x_{n1}(s))h(x_{n1}(s)).
\end{eqnarray*}
From \eqref{4DG}, \eqref{5DG}, \eqref{5bisDG}, \eqref{6DG} we get
\begin{equation}\label{7DG}
\|\alpha_{n1}^{\prime}\|_{\infty}+\|\alpha_{n2}^{\prime}\|_{\infty} \leq C,\  \text {for any } n.
\end{equation}
Hence, from Ascoli-Arzel\`a Theorem we have that, up to subsequences, $\alpha_n$ uniformly converge to some $\alpha\in C^0([t,T],\re^2)$.
In particular, by the definition of $x_{n1}$ and  $x_{n2}$ we get:
\begin{eqnarray*}
&&x_{n1}(s)\to x_1(s)=x_1+ \int_t^s\,\alpha_{1}(\tau)d\tau,\ \text{uniformly in } [t, T],\\
&&x_{n2}(s)\to x_2(s)= x_{2}+\int_t^s\,h(x_{1}(\tau))\alpha_{2}(\tau)d\tau\ \text{uniformly in }[t, T].
\end{eqnarray*}
Moreover, from stability, $\alpha$ is optimal, i.e. $\alpha\in\mathcal U(x,t)$.
From the uniform convergence of the $\alpha_n$ we have in particular that
$\alpha_n(t)\to \alpha(t)$ where $\alpha(t)$ is uniquely determined by assumption. By~\eqref{2DG}, we get $\pi_n\to \pi=\alpha(t)$. This implies that $D_G^*u(x,t)$ is a singleton, then $D_Gu(x,t)$ exists and thank to the first part of the proof $D_Gu(x,t)=-\alpha(t)$.\\
2. If $\mathcal U(x,t)=\{\alpha(\cdot)\}$ then for any $s\in[t,T]$, $\alpha(s)$ is uniquely determined.
Indeed, if there exists $\beta\in \mathcal U(x(s),s)$ the concatenation $\gamma$ of $\alpha$ and $\beta$ (see Proposition \ref{proposition:carloclaudio} in Appendix A) is also optimal, i.e. $\gamma\in\mathcal U(x,t)=\{\alpha(\cdot)\}$.\\
Then applying point 1) with $t=s$,
in $x(s)$ we have that $u$ is $G$-differentiable, i.e. $D_Gu(x(s),s)$ exists.\\
3. From point 1), we know that for any $\alpha(\cdot)\in {\cal{U}}(x,t)$ we have that $\alpha(t)$ is unique
If we know $\alpha(t)$ and that $h(x_1(t))=h(x_1)\neq 0$, then from \eqref{tag:alpha*} we get $p_1(t)$ and $p_2(t)$.
Hence \eqref{tag:1}-\eqref{tag:4} is a system of differential equations with initial conditions $x_i(t)$ and $p_i(t)$, $i=1,2$ which admits an unique solution $(x(s), p(s))$ where $x(s)$ is the unique optimal  trajectory starting from $x$. Moreover still from 1) we have $D_Gu(x,t)=-\alpha(t)$ and from the dynamics \eqref{eq:HJ2} we deduce~\eqref{3.5.3}.
%have $x'_1(t)=-\partial_{x_1} u(x,t)$,$x'_2(t)=-h^2(x_1)\partial_{x_2} u(x,t)$.
\end{proof}

\begin{lemma}\label{B}
%(punto B mail Nicoletta)
Let $x(\cdot):=(x_1(\cdot),x_2(\cdot))$ be an absolutely continuous solution of the problem~\eqref{chflow}
%\begin{equation}\label{4.11}
%\left\{\begin{array}{ll}
%x_1'(s)=-u_{x_1}(x(s),s),& \quad x_1(t)=x_1, \\
%x_2'(s)=-h^2(x_1(s))u_{x_2}(x(s),s),&\quad x_2(t)=x_2,
%\end{array}\right.
%\end{equation}
where $u(x,t)$ is the solution of \eqref{eq:HJ1},
then the control $\alpha=(\alpha_1,\alpha_2)$, with
$$\alpha_1(s)=-u_{x_1}(x(s),s), \ \alpha_2(s)=-h(x_1(s))u_{x_2}(x(s),s)$$ is optimal for $u(x,t)$.
In particular if $u(\cdot, t)$ is G-differentiable at $x$ and $h(x_1)\neq 0$ then problem \eqref{chflow} has a unique solution corresponding to the optimal trajectory.
\end{lemma}
\begin{proof}
We shall adapt the arguments of \cite[Lemma 4.11]{C}. Fix $(t,x)\in(0,T)\times \re^2$ and consider an absolutely continuous solution~$x(\cdot)$ to~\eqref{chflow}; note that this implies that $D_x u$ exists at $(x(s),s)$ for a.e. $s\in (t,T)$. Since $u$ is Lipschitz continuous (see Lemma~\ref{L1}) and $h$ is bounded, also the function~$x(\cdot)$ is Lipschitz continuous and, consequently, also $u(x(\cdot),\cdot)$ is Lipschitz. For a.e. $s\in(t,T)$ there hold:
$i$) $D_x u(x(s),s)$ exists, $ii$) equation \eqref{chflow} holds, $iii$) the function~$u(x(\cdot),\cdot)$ admits a derivative at $s$. Fix such a $s$.

The Lebourg Theorem for Lipschitz function (see \cite[Thm 2.3.7]{Cla90} and \cite[Thm 2.5.1]{Cla90}) ensures that, for any $h\in\R$ small, there exists $(y_h,s_h)$ in the segment $((x(s),s), (x(s+h),s+h))$ and $(\xi^h_x,\xi^h_t) \in co D_{x,t}^*u(y_h,s_h)$ such that
\begin{equation}\label{31}
u(x(s+h),s+h)-u(x(s),s)= \xi^h_x\cdot (x(s+h)-x(s)) +\xi^h_t h
\end{equation}
(here, ``$co$'' stands for the convex hull and $D_{x,t}^*u$ is the Euclidean reachable gradient both in $x$ and in $t$.
%{\it direi convex hull e NON closure of the convex hull come scrive Cardaliaguet})
 The Caratheodory theorem (see \cite[Thm A.1.6]{CS}) guarantees that there exist $(\lambda^{h,i},\xi^{h,i}_x, \xi^{h,i}_t)_{i=1,\dots,4}$ such that $\lambda^{h,i}\geq0$, $\sum_{i=1}^4\lambda^{h,i}=1$, $(\xi^{h,i}_x, \xi^{h,i}_t)\in D_{x,t}^*u(y_h,s_h)$ and $(\xi^h_x,\xi^h_t) = \sum_{i=1}^4\lambda^{h,i}(\xi^{h,i}_x, \xi^{h,i}_t)$.
Note that, as $h\to 0$, $\{\xi^{h,i}_x\}_h$ converge to $D_xu (x(s),s)$  by \cite[Prop 3.3.4-(a)]{CS}; hence also  $\{\xi^{h}_x\}_h$ converge to $D_xu (x(s),s)$ as $h\to 0$.

On the other hand, since $u$ is a viscosity solution to equation~\eqref{eq:HJ1}, by \cite[Proposition II.1.9]{BCD}, we obtain
\[
- \xi^{h,i}_t +\frac12(\xi^{h,i}_{x,1})^2+\frac12h(y_{h,1})^2(\xi^{h,i}_{x,2})^2=f(y_h,s_h);
\]
in particular, as $h\to0$, we deduce
\begin{equation}\label{31bis}
\xi^{h}_t=  \frac12 \sum_{i=1}^4\lambda^{h,i}(\xi^{h,i}_{x,1})^2+\frac12h(y_{h,1})^2\sum_{i=1}^4\lambda^{h,i}(\xi^{h,i}_{x,2})^2 - f(y_h,s_h)\rightarrow \frac12 |D_G u(x(s),s)|^2  - f(x(s),s).
\end{equation}
Dividing \eqref{31} by $h$ and letting $h\to 0$, by equations \eqref{chflow} and \eqref{31bis}, we infer
\begin{eqnarray*}
\frac{d}{ds}u(x(s),s)&=&
D_xu (x(s),s)\cdot x'(s) + \frac12 |D_G u(x(s),s)|^2  - f(x(s),s)\\
&=&-\frac12 |D_G u(x(s),s)|^2 - f(x(s),s)=\frac12 |\alpha|^2  - f(x(s),s)\qquad\textrm{a.e. }s\in(t,T)
\end{eqnarray*}
(recall: $-\alpha=D_Gu(x(s),s)$).
Integrating this equality on $[t,T]$ and taking into account the final datum of~\eqref{eq:HJ1}, we obtain
\[
u(x,t)=\int_t^T\frac12|\alpha|^2+ f(x(s),s) ds +g(x(T)).
\]
Observe that $x(\cdot)$ satisfies the dynamics~\eqref{eq:HJ2} with our choice of~$\alpha(s)$; therefore, the last equality implies that $x(\cdot)$ is an optimal trajectory with optimal control $\alpha(s)=-D_Gu(x(s),s)$.

Let us now prove the last part of the statement.
By Point 3 of Lemma \ref{4.9}, there exists an unique optimal trajectory
$x(\cdot)$ starting from $x$ at time $t$; moreover, by Corollary \ref{cor:nobifurc},  for any $s\in (t,T]$ there exists an unique optimal trajectory starting from $x(s)$ which is the restriction of $x(\cdot)$ to $[s, T]$.
Then, from the representation of the optimal controls \eqref{tag:alpha}, there exists an unique optimal control $\alpha(\cdot)$ and, from points 1 and 2 of Lemma \ref{4.9}, $D_Gu(x(s),s)$ exists and $D_Gu(x(s),s)=-\alpha(s)$, i.e. $x(\cdot)$ is a  solution of \eqref{chflow}.
 Moreover this $x(\cdot)$ is the unique solution still because of Point 3 of Lemma \ref{4.9}.
\end{proof}

\begin{proof}[Proposition \ref{!FP}]
We shall argue following the techniques of \cite[Proposition A.1]{CH} which rely on the Ambrosio superposition principle and on the disintegration of a measure (see \cite{AGS}).
We denote by $\Gamma_T$ the set of continuous curve $C^0([0,T],\re^2)$ and, for any $t\in[0,T]$, we introduce the evaluation map: $e_t: \Gamma_T\to \re^2$ as $e_t(\gamma):=\gamma(t)$. When we say ``for a.e.'' without specifying the measure, we intend w.r.t. the Lebesgue measure.

Let $m\in C^0([0,T],{\mathcal P}_1(\re^2))$ be a solution of problem~\eqref{continuity} in the sense of distributions; in other words, it is a solution to the continuity equation~\eqref{cont}.
We observe that assumption \cite[eq.(8.1.20)]{AGS} is fulfilled because both $Du$ and $h$ are bounded and $m_t:=m(t,\cdot)$ is a measure (see \cite[pag.169]{AGS}); hence we can invoke Ambrosio superposition principle (see \cite[Theorem 8.2.1]{AGS} and also \cite[pag. 182]{AGS}). This principle and the disintegration theorem (see \cite[Theorem 5.3.1]{AGS}) entail that there exist probability measures~$\eta$ and $\{\eta_x\}_{x\in\re^2}$ on~$\Gamma_T$ such that
\begin{equation*}\begin{array}{ll}
i)&e_t\#\eta =m_t \textrm{ and, in particular, } e_0\#\eta =m_0\\
ii)&\eta_x\left(\left\{\gamma\in\Gamma_T: \textrm{$\gamma$ solves \eqref{chflow} with $t=0$ and $x=(x_1,x_2)$}\right\}\right)=1\quad \textrm{for $m_0$-a.e. }x\\
iii)& \eta =\displaystyle\int_{\re^2}\eta_x\, dm_0(x).
\end{array}
\end{equation*}
We recall from assumption (H4) that $m_0$ is absolutely continuous; hence, by assumption (H2) and $\text{meas}\{x\in\re^2 : h(x_1)=0\}=0$,
the optimal synthesis in Lemma~\ref{B} ensures that for a.e. $x\in\re^2$ the solution $\overline{\gamma}_x$ to \eqref{chflow} with $t=0$ and $x=(x_1,x_2)$ is unique and exists because it is the optimal trajectory for the control problem. Therefore, for a.e. $x\in\re^2$, $\eta_x$ coincides with $\delta_{\overline{\gamma}_x}$.
In conclusion, for any function $\phi\in C^0_0(\re^2)$, we have
\begin{eqnarray}
%\label{pushfor}
\notag
\int_{\re^2} \phi\, dm_t&=&\int_{\Gamma_T} \phi( e_t(\gamma)) d\eta(\gamma)=
\int_{\re^2}\left(\int_{e_0^{-1}(x)}\phi( e_t(\gamma)) d\eta_x(\gamma)\right)\, dm_0(x)\\ \notag
%\label{ambrosio}
&=&\int_{\re^2}\phi(\overline{\gamma}_x(t))m_0(x)\, dx.
\end{eqnarray}
Since the integrand in the last term is uniquely defined up to a set of null measure, also the first term is uniquely defined; consequently, $m$ is uniquely defined.
\end{proof}
In the following corollary we use the previous characterization to prove the Lipschitz regularity of $m$.
\begin{corollary}\label{lemma:m_lip}
The unique bounded solution $m$
%\in C^0([0,T];\mathcal P_1(\R^2))$
to problem~\eqref{continuity} is a Lipschitz continuous map from $[0,T]$ to $\mathcal P_1(\R^2)$ with a Lipschitz constant bounded by $\|Du\|_\infty \|h^2\|_\infty$.
\end{corollary}
\begin{proof}
Let $m$ be the unique solution to problem \eqref{continuity} as in Proposition~\ref{VV} and Proposition~\ref{!FP}. Fix $\phi$, a $1$-Lipschitz continuous function on $\re^2$.
By relation~\eqref{ambrosio2}, for any $t_1,t_2\in[0,T]$, we infer
\begin{eqnarray*}
\int_{\re^2} \phi\, dm_{t_1}-\int_{\re^2} \phi\, dm_{t_2}&=&
\int_{\re^2}\phi(\overline{\gamma}_x(t_1))-\phi(\overline {\gamma}_x(t_2))m_0(x)\, dx\\
&\leq&\int_{\re^2}\left|\overline{\gamma}_x(t_1)- \overline {\gamma}_x(t_2)\right|m_0(x)\, dx\\
&\leq& \|Du\|_\infty \|h^2\|_\infty|t_1-t_2|
\end{eqnarray*}
where the last relation is due to the definition of $\overline\gamma$ as solution to problem~\eqref{chflow} and to the boundedness of $Du$ and of $h$. Hence, passing to the $\sup_{\phi}$ in the previous inequality, the Kantorovich-Rubinstein theorem (see \cite[Remark 6.5]{V} or \cite[Theorem 5.5]{C}) ensures
\[
{\bf d}_1(m_{t_1},m_{t_2})\leq \|Du\|_\infty \|h^2\|_\infty|t_1-t_2|.
\]
\end{proof}

\begin{proof}[Theorem~\ref{prp:m}]{\empty}
The existence of $m$ follows from Proposition~\ref{VV}, the uniqueness and the representation formula comes from Proposition~\ref{!FP} and the Lipschitz regularity is proved in Corollary~\ref{lemma:m_lip} here above.
\end{proof}

\section{Proof of the main Theorem}\label{sect:MFG}
This section is devoted to the proof of our main Theorem~\ref{thm:main}.

\begin{proof}[Theorem~\ref{thm:main}]{\empty}\\
1.
  We shall argue following the proof of \cite[Theorem 4.1]{C} (see also \cite{LL1,LL2,LL3}).
Consider the set ${\cal C} :=\{m\in C^0([0,T], {{\cal P}_1})\mid m(0)=m_0\}$ and observe that it is convex. We also introduce a map $T:{\cal C}\rightarrow {\cal C}$ as follows: to any $m\in {\cal C}$ we associate the solution~$u$ to problem~\eqref{eq:HJ1} with $f(x,t)=F(x,m)$ and $g(x)=G(x,m(T))$ and to this $u$ we associate the solution~$\mu=:T(m)$ to problem \eqref{continuity}.
By a stability result proved in \cite[Lemma 4.19]{C}), the map~$T$ is continuous. Moreover, Corollary ~\ref{lemma:m_lip} (note that the constant is independent of~$m$) implies that the map $s\rightarrow T(m)(s)$ is uniformly Lipschitz continuous with value in the compact set of measures on a compact set (still independent of~$m$); hence, the map~$T$ is compact.
Invoking Schauder fix point Theorem, we accomplish the proof of (i).\\
2.
Theorem \ref{prp:m} ensures that, if $(u,m)$ is a solution of \eqref{eq:MFG1},
 for any function $\phi\in C^0_0(\re^2)$, we have
\begin{equation}\label{reprfor}
\int_{\re^2} \phi\, dm(t)=\int_{\re^2}\phi(\overline{ \gamma}_x(t))m_0(x)\, dx
\end{equation}
where $\overline{\gamma}_x$ is the solution of \eqref{chflow} (with $t=0$ and $x=(x_1,x_2)$) and it is uniquely defined for a.e. $x\in\R^2$. The last relation is equivalent to the statement.
\end{proof}
\begin{remark}
As in \cite[Theorem 4.20]{C} also the vanishing viscosity method may be applied to prove the
existence of a solution of system \eqref{eq:MFG1}.
%as stated in  part 1 of Theorem \ref{thm:main}.
Actually, it suffices to follow the same arguments of Section~\ref{subsect:ex} with $F(x,\overline m)$ and $G(x,\overline m(T))$ replaced respectively by $F(x, m^\sigma)$ and $G(x,m^\sigma(T))$. Note also that Lemma~\ref{visco:lemma4} ensures that the function $m^\sigma$ fulfills the assumption~\eqref{mcnd}.
Because of the degenerate term $h$, we cannot directly deduce the representation formula \eqref{reprfor} invoking the results in \cite{C},  but we can apply the results of Section \ref{uniq}.
%
%The second part of Theorem \ref{thm:main}, instead, cannot be deduced by the methods of \cite{C}
%due to the fact that, because of the presence of the degenerate term $h$, we do not have the uniqueness of the optimal trajectory after the initial time (see Theorem \ref{th:nobifurc}).
\end{remark}
\section{Appendix}
\subsection{A- Concatenation of optimal trajectories and the Dynamic Programming Principle}
\begin{definition}\label{def:concatenation}
For $0\leq t\leq r<T$, let $\varphi:[t,T]\to\R^n$ and $\psi :[r,T]\to\R^n$. The concatenation of $\varphi$ with $\psi$ at $r$ is the function $\xi:[t,T]\to \R^n$ defined by
\[\xi =\varphi\text{ on }[t,r],\qquad \xi =\psi\text{ on }[r, T].\]
%\[\xi =\begin{cases}\varphi\text{ on }[t,r]\\
%\psi\text{ on }[r, T].\end{cases}\]
\end{definition}
%\begin{remark} {\huge C.: ma questo remark serve veramente?} We do not require, in Definition~\ref{def:concatenation}, that the functions $\varphi, \psi$ here involved are continuous. If this is the case, and $\varphi(r)=\psi(r)$, then the resulting function $\xi$ turns out to be continuous.
%\end{remark}
The following variant of the Dynamic Programming Principle will be used in the sequel. The arguments of Point 4 are similar to those employed in  \cite[Proposition III.2.5]{BCD}.
\begin{proposition}[Dynamic Programming Principle]\label{proposition:carloclaudio} Let $x^*$ be optimal for $u(x,t)$, and $r\in [t, T]$. Let $\alpha^*$ be optimal control for $x^*$.
\begin{enumerate}
\item  Let $y^*$ be optimal for $u(x^*(r),r)$. The concatenation of $x^*$ with $y^*$ at $r$ is optimal for $u(x,t)$ and, moreover,
\begin{equation}u( x, t)=u(x^*(r),r)+\int_t^r\dfrac12|\alpha^*(s)|^2+f(x^*(s),s)\,ds;\label{tag:DPP1}\end{equation}
\item The trajectory $x^*$, restricted to $[r, T]$, is optimal for $u(x^*(r),r)$;
\item The couple $(x^*, \alpha^*)$, restricted to $[t, r]$, is optimal for the following optimal control problem with prescribed endpoints:
\[
\text{Minimize } \displaystyle I_{t,r}(x,\alpha):=\int_t^r\dfrac12|\alpha(s)|^2+f( x(s),s)\,ds,\]
with $(x(\cdot),\alpha)$ subject to \eqref{eq:HJ2} and $x(r)= x^*(r)$.
\item The Dynamic Programming Principle holds:
\begin{equation}\label{tag:DPP2}u( x, t)=\min_{(x(\cdot),\alpha)\in\mathcal A(x,t)}\left\{u(x(r), r)+\int_t^r\dfrac12|\alpha(s)|^2+f(x(s),s)\,ds\right\}.\end{equation}
\end{enumerate}
\end{proposition}
\begin{proof} % Let $u( x,t)=J_t(x^*, \alpha^*)$.
1. Let $ \beta^*$ be optimal control for $y^*$. Let $(z^*,\gamma^*)$ be the concatenation of $(x^*,\alpha^*)$ with $(y^*,\beta^*)$ at $r$: clearly $(z^*, \gamma^*)$ is admissible for {\rm (OC)} of Definition \ref{def:OCD}.
The minimality of $(x^*,\alpha^*)$ for $u(x,t)$, and that of $(y^*,\beta^*)$ for $u(x^*(r),r)$, directly yield
\[\begin{aligned}u( x, t)&=\int_t^r\dfrac12|\alpha^*|^2+f(x^*,s)ds+
\left(\int_r^T\dfrac12|\alpha^*|^2+f(x^*,s)ds+g(x^*(T))\right)\\
&\ge \int_t^r\dfrac12|\alpha^*|^2+f(x^*,s)ds+u(x^*(r),r)\\
 &= \int_t^r\dfrac12|\alpha^*|^2+f(x^*,s)ds+
\left(\int_r^T\dfrac12|\beta^*|^2+f(y^*,s)ds+g(y^*(T))\right)\\
&=J_t(z^*, \gamma^*)\ge u( x, t),\end{aligned}\]
so that the above inequalities are actually equalities, proving \eqref{tag:DPP1} and the optimality of $(z^*, \gamma^*)$.

2. Let $(y, \beta)$ be admissible for $u(x^*(r),r)$. Let $(z,\gamma)$ be the concatenation of $(x^*,\alpha^*)$ with $(y,\beta)$ at $r$. The conclusion follows from the following inequality:
\[\begin{aligned}
0\le J_t(z, \gamma)-J_t(x^*, \alpha^*)=J_r(y, \beta)-J_r(x^*, \alpha^*).
\end{aligned}\]

3. Assume that $(x(\cdot),\alpha)$ is admissible for $u(x,t)$, in the interval $[t,r]$, i.e., satisfies \eqref{eq:HJ2} together with the \emph{endpoint} condition $x(r)= x$. Then the concatenation $(z,\gamma)$ of $(x,\alpha)$ with $(x^*, \alpha^*)$, restricted to $[r,T]$,  at $r$ is admissible. The minimality of $(x^*, \alpha^*)$ implies that
\begin{equation}\label{tag:ineqI}J_t(x^*, \alpha^*)\le J_t(z,\gamma).\end{equation}
Now
\[J_t(x^*, \alpha^*)=I_{t,r}(x^*,\alpha^*)+J_r(x^*, \alpha^*),\quad J_t(z, \gamma)=I_{t,r}(x(\cdot),\alpha)+J_r(x^*, \alpha^*).
\]
It follows from \eqref{tag:ineqI} that $I_{t,r}(x^*,\alpha^*)\le I_{t,r}(x(\cdot),\alpha)$.

4.
Let $(x(\cdot),\alpha)$ be admissible and $(y^*,\beta^*)$ be optimal for $u(x(r),r)$. Let $(z,\gamma)$ be the concatenation of $(x(\cdot),\alpha)$ with $(y^*,\beta^*)$ at $r$. Since $(z,\gamma)$ is admissible we get
\[u( x, t)\le J_t(z,\gamma)=\int_t^r\dfrac12|\alpha(s)|^2+f(x(s),s)\,ds+u(x(r),r),
\]
proving that
\[u( x, t)\le \min_{(x(\cdot),\alpha)\in\mathcal A(x,t)}\left\{u(x(r), r)+\int_t^r\dfrac12|\alpha(s)|^2+f(x(s),s)\,ds\right\}.\]
The opposite inequality follows from \eqref{tag:DPP1}.
\end{proof}
\subsection{Appendix B: G-differentials}
\label{sect:Gdiff}
In this section, we introduce the notion of $G$-differentiability and we collect several properties of semiconcave functions. %Let $v=(v_1,v_2)$ and $\theta=(\theta_1,\theta_2)$ be.
\begin{definition}
\label{G-differenz}
A function~$u:\re^2\rightarrow \re$ is $G$-differentiable in $x\in\re^2$ if there exists $p_G\in \re^2$ such that
\[
\lim_{v\rightarrow 0}\frac{u(x_1+v_1,x_2+h(x_1)v_2)-u(x_1, x_2)-(p_G,v)}{\vert v\vert}=0;
\]
in this case we denote $p_G=D_Gu(x)$.
We define the $G$-subdifferential
\begin{eqnarray*}
D^{-}_Gu(x) &:=& \{p\in \re^2\vert
 \liminf_{v\rightarrow 0}\frac{u(x_1+v_1,x_2+h(x_1)v_2)-u(x_1, x_2)-(p,v)}{\vert v\vert}\geq 0\},%\\
%D^{+}_Gu(x)&:=& \{p\in \re^2\vert \limsup_{v\rightarrow 0}\frac{u(x_1+v_1,x_2+h(x_1+v_1)v_2)-u(x_1, x_2)-(p,v)}{\vert v\vert}\leq 0\}.
\end{eqnarray*}
the lower $G$-Dini derivative in the direction $\theta$ (i.e., $|\theta| =1$)
\begin{eqnarray*}
\partial_G^{-}u(x,\theta) &:=& \liminf_{l\to 0^+,\theta'\to\theta}\frac{u(x_1+l\theta'_1,x_2+h(x_1)l\theta'_2)-u(x_1, x_2)}{l}
\end{eqnarray*}
and the generalized $G$-lower derivative in the direction~$\theta$
\[
u^0_{G,-}(x,\theta):=\liminf_{l\to 0^+, y\to x} \frac{u(y_1+l\theta_1,y_2+h(y_1)l\theta_2)-u(y_1,y_2)}{l}.
\]
The $G$-superdifferential $D^{+}_Gu(x)$, the upper $G$-Dini derivative~$\partial_G^{+}u(x,\theta)$ and the generalized $G$-upper derivative $u^0_{G,+}(x,\theta)$ are defined in an analogous way.
We introduce the reachable $G$-gradients
\[
D_G^*u(x):=\{p\,:\, \exists x_n\rightarrow x, \textrm{ $u$ is $G$-differentiable at $x_n$ and $D_Gu(x_n)\rightarrow p$}\}.
\]
We define the ($1$-sided) $G$-directional derivative of $u$ at $x$ in the direction $\theta$ as
\[
\partial_Gu(x,\theta):=\lim_{l\to 0^+}\frac{u(x_1+l\theta_1,x_2+h(x_1)l\theta_2)-u(x_1, x_2)}{l}.
\]
\end{definition}
\begin{lemma}\label{D+D-}
\begin{enumerate}
\item If $u$ is $G$-differentiable at $x$, then $D_Gu(x)$ is unique and $D^+_Gu(x)$ and $D^-_Gu(x)$ are both nonempty.
\item For $h(x_1)\ne 0$, there holds: $(p_1,p_2)\in D^+u(x)$ if and only if $(p_1,h(x_1)p_2)\in D^+_Gu(x)$.
\item For $h(x_1)= 0$ and $|\theta|=1$, there holds:
\begin{eqnarray*}
D^+_Gu(x)&=&\{(p_1,0)\,:\,\limsup_{v_1\to 0} \frac{u(x_1+v_1,x_2)-u(x_1, x_2)-p_1v_1}{\vert v_1\vert}\leq 0 \}\\
\partial_G u(x,\theta)&=&\left\{\begin{aligned}&0&\textrm{ for $\theta_1=0$}\\
&|\theta_1| \partial u(x,(\textrm{sgn}(\theta_1),0))&\textrm{ for $\theta_1\ne0$}\end{aligned} \right.
\end{eqnarray*}
where $\partial u(x,\theta)$ is the standard directional derivative of $u$ at $x$ in the direction~$\theta$.
%As in~\cite[Prop.3.1.5-(c)]{CS}, one can easily prove that, if~$D^{-}_Gu(x)$ and~$D^{+}_Gu(x)$ are both not empty, then~$u$ is $G$-differentiable with~$D_Gu(x)=D^{-}_Gu(x)=D^{+}_Gu(x)$.
\item For Lipschitz continuous function $u$, there holds:
\begin{equation}
\partial_G^{-}u(x,\theta) := \liminf_{l\to 0^+}\frac{u(x_1+l\theta_1,x_2+h(x_1)l\theta_2)-u(x_1, x_2)}{l},\label{tag:d1}
\end{equation}
\begin{equation}
\text{If }h(x_1)=0\text{ then }\qquad (p_1,p_2)\in D_G^*u(x)\Rightarrow p_2=0.\label{tag:d2}
\end{equation}
%\begin{eqnarray*}
%&&\partial_G^{-}u(x,\theta) := \liminf_{l\to 0^+}\frac{u(x_1+l\theta_1,x_2+h(x_1)l\theta_2)-u(x_1, x_2)}{l}\\
%&&\text{If }h(x_1)=0\text{ then }\qquad (p_1,p_2)\in D_G^*u(x)\Rightarrow p_2=0.
%\end{eqnarray*}
\end{enumerate}
\end{lemma}
\begin{proof} Points 1, 2 and 3 are obvious. The equality in \eqref{tag:d1}  follows by the arguments of \cite[Remark 3.1.4]{CS}. Let us prove \eqref{tag:d2}. For any $(p_1,p_2)\in D_G^*u(x)$, there exists $\{x_k\}_k$ with $x_k:=(x_{k,1}, x_{k,2})\to x$ and $D_Gu(x_k)\to (p_1,p_2)$. Possibly passing to a subsequence, we may assume that either $h(x_{k,1})\ne 0$ for any $k$ or $h(x_{k,1})= 0$ for any $k$. In the first case, by Point 2, we have $D_Gu(x_k)=(D_1u(x_k), h(x_{k,1})D_2u(x_k))$ where $D_1$ and $D_2$ are the partial derivatives with respect to $x_1$ and $x_2$. As $k\to +\infty$, by the Lipschitz continuity of $u$, we get $p_2=\displaystyle\lim_k h(x_{k,1})D_2u(x_k)=0$. In the latter case, $D_Gu(x)=(D_1u(x_{k,1}),0)\to (p_1,0)$, the conclusion follows.
\end{proof}
\begin{proposition}\label{315}
We have
\[D^+_Gu(x)=\{p:\, \partial^+_Gu(x,\theta)\leq (p,\theta)\, \forall \theta\in\re^2\},\quad
D^-_Gu(x)=\{p\, :\, \partial^-_Gu(x,\theta)\geq (p,\theta)\,\forall \theta\in\re^2\}.
\]
Moreover, $D^+_Gu(x)$ and $D^-_Gu(x)$ are both nonempty if and only if $u$ is $G$-differentiable at $x$ and in this case they reduce to the singleton $D_Gu(x)=D^-_Gu(x)=D^+_Gu(x)$.
\end{proposition}
The proof of this proposition follows the same arguments of \cite[Proposition 3.1.5]{CS}; actually the main difference is that one has to consider $x_k=(x_1+v_{k,1}, x_2+h(x_1)v_{k,2})$ with
$v_k=(v_{k,1}, v_{k,2})\to 0$. Hence we shall omit it.
\begin{proposition}\label{prp:gdiff}
Let $u$ be a semiconcave function with modulus of semiconcavity~$\omega$. Then there hold
\begin{enumerate}
\item $p\in D^+_G u(x)$ if and only if for any $v=(v_1,v_2)\in \R^2$
\begin{equation}\label{CNESGsemiconcavita}
u(x_1+v_1,x_2+h(x_1)v_2)-u(x_1, x_2)-(p,v)\leq |(v_1, h(x_1)v_2)| \omega(|(v_1, h(x_1)v_2)|);
\end{equation}
\item If $\lim_k x_k=x$ and $p_k\in D^{+}_Gu(x_k)$ with $\lim_k p_k=p$, then $p\in D^{+}_Gu(x)$; hence, $D^*_Gu(x)\subset D^+_Gu(x)$;
\item  $D^{+}_Gu(x,t)\ne\emptyset$;
%\item $D^{+}_Gu(x,t)$ is the closed convex hull of the set of the reachable $G$-gradients
%\begin{equation}
%\label{DGstella}
%D_G^*u(x,t)=\{p\,:\, \hbox{ s.t. } \exists x_n\rightarrow x \hbox{ where }  \exists D_G(x_n,t):=p_n  \hbox{ and } p_n\rightarrow p\}
%\end{equation}
\item If $D^{+}_Gu(x)=\{p\}$ (i.e., it is a singleton), then $u$ is $G$-differentiable at~$x$.
\end{enumerate}
\end{proposition}
\begin{proof}
1. Consider $p\in D^{+}_Gu(x)$. When $h(x_1)=0$ and $v_1=0$, inequality~\eqref{CNESGsemiconcavita} is a trivial consequence of Point 3 of Lemma~\ref{D+D-}. Otherwise, the rest of the proof is an adaptation of the argument in~\cite[Proposition 3.3.1]{CS} using \cite[equation (2.1)]{CS} with $y=(x_1+v_1,x_2+h(x_1)v_2)$.\\
2. It follows directly from~\eqref{CNESGsemiconcavita}.\\
3. Being semiconcave, the function~$u$ is locally Lipschitz continuous. By Rademacher's theorem, there exists a sequence of points $\{x_k\}_k$ with $\lim_k x_k=x$ where $u$ is differentiable and, in particular, $G$-differentiable with $\vert D_Gu(x_k)\vert\leq L$ (for some~$L$). Possibly passing to a subsequence, $D_Gu(x_k)\rightarrow p$; hence, by point (2), $p\in D^{+}_Gu(x)$.\\
4. By Proposition~\ref{315}, it suffices to prove: $p\in D^{-}_Gu(x)$.
To this end, consider any sequence $\{v_k\}_k$, with $v_k\to 0$ as $k\to+\infty$ and introduce $\{x_k\}_k$ as
\[
x_k=(x_{k,1}, x_{k,2}):=(x_1+v_{k,1}, x_2+h(x_{k,1})v_{k,2}).
\]
We observe that: ($i$) $x_k\to x$ as $k\to+\infty$, ($ii$)  by point~(3), $\exists p_k\in D^+_Gu(x_k)$ with $|p_k|\leq L$, ($iii$) by point~(2) and possibly passing to a subsequence, $p_k\rightarrow p$ as $k\to+\infty$.
Relation~\eqref{CNESGsemiconcavita} centered in $x_k$ defined above,  with $v=-v_k$, gives
\begin{multline*}
-u(x_{k,1}-v_{k,1},x_{k,2}-h(x_{k,1})v_{k,2})+u(x_{k,1}, x_{k,2})-(p_k,v_k)\\ \geq - |(v_{k,1}, h(x_{k,1})v_{k,2})| \omega(|(v_{k,1}, h(x_{k,1})v_{k,2})|).
\end{multline*}
By our choice of $x_k$, this inequality entails
\begin{equation*}
\begin{aligned}
&\frac{-u(x_{1},x_{2})+u(x_1+v_{k,1}, x_2+h(x_{1})v_{k,2})-(p,v_k)}{|v_k|} \\
&\geq \frac{u(x_1+v_{k,1}, x_2+h(x_{1})v_{k,2})-u(x_1+v_{k,1}, x_2+h(x_{k,1})v_{k,2})+(p_k-p,v_k)}{|v_k|}\\&\qquad - \frac{|(v_{k,1}, h(x_{k,1})v_{k,2})| \omega(|(v_{k,1}, h(x_{k,1})v_{k,2})|)}{|v_k|}\\
&\geq \frac{L L' |v_{k,2}||v_{k,1}|}{|v_k|}+ (p_k-p,v_k/|v_k|)- \frac{|(v_{k,1}, h(x_{k,1})v_{k,2})| \omega(|(v_{k,1}, h(x_{k,1})v_{k,2})|)}{|v_k|}
\end{aligned}
\end{equation*}
where $L$ and $L'$ are respectively local Lipschitz constants of $u$ and of $h$. Letting $k\to+\infty$, we obtain
\[
\liminf_{k\to+\infty}\frac{u(x_1+v_{k,1}, x_2+h(x_{1})v_{k,2})-u(x_{1},x_{2})-(p,v_k)}{|v_k|}\geq 0;
\]
by the arbitrariness of $v_k$, we conclude: $p\in D^{-}_Gu(x)$.
\end{proof}
In the next statement we establish that semiconcave functions always have directional derivatives.
\begin{proposition}
Let $u$ be a semiconcave function with modulus of semiconcavity~$\omega$. Then, for any direction~$\theta$, the directional derivative $\partial_Gu (x,\theta)$ exists and the following equalities hold:
\[\partial_Gu (x,\theta)=\partial^-_Gu (x,\theta)=\partial_G^+u (x,\theta)=u^0_{G,-}(x,\theta).
\]
%it coincides with $\partial^-_Gu (x,\theta)$, $\partial_G^+u (x,\theta)$ and $u^0_{G,-}(x,\theta)$.
\end{proposition}
\begin{proof} The proof is similar to the proof of \cite[Theorem 3.2.1]{CS} so we just sketch it. Fix a direction~$\theta$ and consider $0<l_1<l_2$. Relation~\cite[eq. (2.1)]{CS} with $\lambda=1-l_1/l_2$, $y=(x_1+l_2\theta_1,x_2+h(x_1)l_2\theta_2)$ entails
\begin{multline}\label{316}
\frac{u(x_1+l_1\theta_1, x_2+h(x_1)l_1\theta_2)-u(x)}{l_1}\geq
\frac{u(x_1+l_2\theta_1, x_2+h(x_1)l_2\theta_2)-u(x)}{l_2} \\-\left(1-\frac{l_1}{l_2}\right)|(\theta_1, h(x_1)\theta_2)| \omega(l_2|(\theta_1, h(x_1)\theta_2)|).
\end{multline}
Passing to the $\liminf_{l_1\to 0^+}$ and after to the $\limsup_{l_2\to 0^+}$, we get $\partial ^-_G u(x,\theta)\geq \partial^+_Gu(x,\theta)$; hence, $\partial_G u(x,\theta)$ exists and it coincides both with the upper and the lower $G$-Dini derivatives.
Moreover, by the definitions of $\partial_G^+ u(x,\theta)$ and of $u^0_{G,-} (x,\theta)$, Point 4 of Lemma~\ref{D+D-} easily entails: $\partial_G^+ u(x,\theta)\geq u^0_{G,-} (x,\theta)$. Therefore, it remains to prove
\begin{equation}\label{317}
\partial_G^+ u(x,\theta)\leq u^0_{G,-} (x,\theta).
\end{equation}
%We argue as in \cite[Theorem 3.2.1]{CS}.
Let $\epsilon$ and $\overline \ell$ be two fixed positive constants with $\overline \ell\geq l$. Since $u$ is continuous, there exists $\alpha$ sufficiently small such that
\[
\frac{u(x_1+\overline \ell\theta_1,x_2+\overline \ell\theta_2h(x_1))-u(x)}{\overline \ell}\leq
\frac{u(y_1+\overline \ell\theta_1,y_2+\overline \ell\theta_2h(y_1))-u(y)}{\overline \ell}+\epsilon\qquad \forall y\in B_\alpha(x).
\]
By inequality~\eqref{316} (with $x$, $l_1$ and $l_2$ replaced respectively by $y$, $l$), we get
\begin{multline*}
\frac{u(y_1+\overline \ell\theta_1,y_2+\overline \ell\theta_2h(y_1))-u(y)}{\overline \ell}\leq
\frac{u(y_1+l\theta_1,y_2+l\theta_2h(y_1))-u(y)}{l}\\
+\frac{\overline \ell-l}{\overline \ell} |(\theta_1,h(y_1)\theta_2)| \omega(\overline \ell|(\theta_1,h(y_1)\theta_2)|)\qquad \forall l\in(0,\overline \ell).
\end{multline*}
By the last two inequalities we deduce
\begin{multline*}
\frac{u(x_1+\overline \ell\theta_1,x_2+\overline \ell\theta_2h(x_1))-u(x)}{\overline \ell}\leq
\min\limits_{y\in B_\alpha(x),l\in(0,\overline \ell)}\frac{u(y_1+l\theta_1,y_2+l\theta_2h(y_1))-u(y)}{l}\\+|(\theta_1,h(y_1)\theta_2)| \omega(\overline \ell|(\theta_1,h(y_1)\theta_2)|)+\epsilon.
\end{multline*}
Taking into account the definition of $u^0_{G,-}(x,\theta)$, we get
\[
\frac{u(x_1+\overline \ell\theta_1,x_2+\overline \ell\theta_2h(x_1))-u(x)}{\overline \ell}\leq
u^0_{G,-}(x,\theta)+|(\theta_1,h(x_1)\theta_2)| \omega(\overline \ell|(\theta_1,h(x_1)\theta_2)|)+\epsilon.
\]
In conclusion, passing to the limit for $\epsilon\to 0^+$ and then
$\displaystyle\limsup_{\overline \ell \to 0}$, we obtain inequality~\eqref{317}.
\end{proof}
\begin{theorem}\label{thm336}
Let $u$ be a semiconcave function. Then, there holds
\begin{equation}\label{s11}
D^+_Gu(x)= co D^*_Gu(x);
\end{equation}
moreover, for any direction~$\theta$, the $G$-directional derivative of~$u$ in the direction~$\theta$ satisfies
\begin{equation}\label{2s11}
\partial_Gu(x,\theta)=\min_{p\in D^+_Gu(x)}(p,\theta)=\min_{p\in D^*_Gu(x)}(p,\theta).
\end{equation}
\end{theorem}
\begin{proof}
We shall use some of the arguments of \cite[Theorem 3.3.6]{CS}. Let us prove relations~\eqref{2s11}. For any direction~$\theta$, using Proposition~\ref{315} and Proposition~\ref{prp:gdiff}-(2), we obtain
\[
\partial_Gu(x,\theta)\leq \min_{p\in D^+_Gu(x)}(p,\theta)\leq \min_{p\in D^*_Gu(x)}(p,\theta).
\]
Hence, it remains to prove
\begin{equation}\label{3s11}
\min_{p\in D^*_Gu(x)}(p,\theta)\leq \partial_Gu(x,\theta)\qquad\textrm{for any direction~$\theta$}.
\end{equation}
In order to prove this inequality, we study separately the cases when~$x_1$ belongs or not to $\{h(x_1)=0\}$. Assume $h(x_1)\ne 0$ and fix a direction~$\theta$. Since $u$ is differentiable a.e., there exists a sequence~$\{v_k\}_k$, with $v_k\in\re^2$, such that: $(i)$ $v_k\to 0$ as $k\to+\infty$, $(ii)$ $v_k/|v_k|\to \theta$ as $k\to+\infty$, $(iii)$ $u$ is differentiable at~$x_k:=(x_1+v_{k,1}, x_2+v_{k,2}h(x_1))$, $(iv)$ (taking advantage of the Lipschitz continuity of $u$ and possibly passing to a subsequence) $D_Gu(x_k)$ converge to some $p\in D^*_Gu(x)$ as $k\to+\infty$.
Applying inequality~\cite[eq. (3.18)]{CS} (with $x$ and $y$ replaced respectively by $x_k$ and $x$), we get
\begin{equation}\label{tag:star1}
u(x)-u(x_k)+(Du(x_k),(v_{k,1},h(x_1)v_{k,2}))\leq |(v_{k,1},h(x_1)v_{k,2})|\omega(|(v_{k,1},h(x_1)v_{k,2})|).
\end{equation}
On the other hand, we observe that point $(iii)$ here above and Point 2 of Lemma~\ref{D+D-} ensure that~$u$ is $G$-differentiable at~$x_k$ with $D_Gu(x_k)=(D_1u(x_k), h(x_{k,1})D_2u(x_k))$.
Hence, we have
\begin{equation}\label{tag:star2}
\begin{aligned}(Du(x_k),(v_{k,1},h(x_1)v_{k,2}))&=(D_Gu(x_k),v_{k}) +D_2u(x_k)v_{k,2}[h(x_1)-h(x_{k,1})]\\
&\geq  (D_Gu(x_k),v_{k}) - C|v_{k,2}| |v_{k,1}|,
\end{aligned}
\end{equation}
where the last inequality holds for a suitable $C>0$, and is due to the Lipschitz continuity of~$u$ and of~$h$.
%\end{proof}
By \eqref{tag:star1} and \eqref{tag:star2}, we get
\[
(D_Gu(x_k),v_{k}/|v_{k}|)\leq \frac{u(x_k)-u(x)}{|v_{k}|}+\frac{C|v_{k,2}| |v_{k,1}|}{|v_{k}|}+\frac{|(v_{k,1},h(x_1)v_{k,2})|}{|v_{k}|}\omega(|(v_{k,1},h(x_1)v_{k,2})|).
\]
Letting $k\to+\infty$, we infer: $(p,\theta)\leq \partial_Gu(x,\theta)$ for some $p\in D_G^*u(x)$ which, in turns, entails~\eqref{3s11}.

Consider now $x$ such that $h(x_1)=0$. By Point 4 of Lemma \ref{D+D-} we have: $\min_{p\in D^*_G u(x)}(p,\theta) = \min_{p\in D^*_G u(x)} p_1\theta_1$; taking into account also Point 3 of Lemma \ref{D+D-},
%we have respectively
%\[
%\min_{p\in D^*_G u(x)}(p,\theta) = \min_{p\in D^*_G u(x)} p_1\theta_1\qquad\textrm{and}\qquad \partial_Gu(x,\theta)=\theta_1 D_1u(x).
%\]
%Therefore, in this case,
relation~\eqref{3s11} is equivalent to
\[
\min_{p\in D^*_G u(x)} p_1\textrm{sgn}(\theta_1)\leq  \partial u(x,(\textrm{sgn}(\theta_1),0))\qquad\forall \theta_1\in[-1,1]\setminus\{0\}.
\]
In order to prove this relation, we follow an argument similar to the previous case. We consider a sequence $\{v_k\}_k$ such that: $(i)$ $v_k\to 0$ as $k\to+\infty$, $(ii)$ $v_k/|v_k|\to (\textrm{sgn}(\theta_1),0)$ as $k\to+\infty$  (in particular $v_{k,2}/|v_k|\to 0$), $(iii)$ $u$ is differentiable at~$x_k:=(x_1+v_{k,1}, x_2+v_{k,2})$ (note that this definition is different from the corresponding one in the previous case), $(iv)$ $D_Gu(x_k)$ converge to some $p\in D^*_Gu(x)$ as $k\to+\infty$.
Applying inequality~\cite[eq. (3.18)]{CS} (with $x$ and $y$ replaced respectively by $x_k$ and $x$), we get
\[
(Du(x_k),v_{k})\leq u(x_k)-u(x)+ |v_{k}|\omega(|v_{k}|).
\]
Again we get that $u$ is $G$-differentiable at~$x_k$ with $D_Gu(x_k)=(D_1u(x_k), h(x_{k,1})D_2u(x_k))$. Hence, we deduce
\[
(Du(x_k), v_k)=(D_Gu(x_k),v_k)+D_2u(x_k)[1-h(x_{k,1})]v_{k,2}\geq (D_Gu(x_k),v_k)-C|v_{k,2}|
\]
where the last inequality is due to the Lipschitz continuity of~$u$ and to the boundedness of~$h$.
By the last two inequalities, we get
\[
(D_Gu(x_k),v_{k}/|v_{k}|)\leq \frac{u(x_k)-u(x)}{|v_{k}|}+\frac{C|v_{k,2}|}{|v_{k}|}+\omega(|v_k|).
\]
Letting $k\to+\infty$, we infer: $p_1\textrm{sgn}(\theta_1)\leq \partial u(x,(\textrm{sgn}(\theta_1),0))$.
Hence, relations~\eqref{2s11} are completely proved. Arguing as in~\cite[Theorem 3.3.6]{CS}, we infer relation~\eqref{s11}.
\end{proof}

\noindent{\bf Acknowledgments.} The first and the second authors are members of GNAMPA-INdAM and were partially supported also by the research project of the University of Padova ``Mean-Field Games and Nonlinear PDEs'' and by the Fondazione CaRiPaRo Project ``Nonlinear Partial Differential Equations: Asymptotic Problems and Mean-Field Games''. The fourth author has been partially funded by the ANR project ANR-16-CE40-0015-01.
The authors wish to warmly thank  P. Cardaliaguet for helpful discussions and hints.

\end{document}